\documentclass[options]{amsart}
\usepackage{latexsym,ifthen,amssymb}
\usepackage[toc,page,title,titletoc,header]{appendix}
\usepackage{enumerate}
\usepackage{graphicx}
\usepackage{bm}
\usepackage{subfigure}
\usepackage{float}
\usepackage{rotating}
\usepackage{epstopdf}
\usepackage{hyperref}

\usepackage{bm}

\usepackage[toc,page,title,titletoc,header]{appendix}
\usepackage{multirow}
 \usepackage{longtable}
 \usepackage{rotating}
 \usepackage{thmtools}
\usepackage{thm-restate}

\newtheorem{theorem}{Theorem}[section]
\newtheorem{lemma}[theorem]{Lemma}

\newtheorem{claim}[theorem]{Claim}

\theoremstyle{definition}
\newtheorem{definition}[theorem]{Definition}

\theoremstyle{remark}
\newtheorem{remark}[theorem]{Remark}

\numberwithin{equation}{section}

\newtheoremstyle{noparens}%
  {}{}%
{}{}%
{\bfseries}{.}%
{ }%
{\thmname{#1}\thmnumber{ #2}\thmnote{ #3}}
\theoremstyle{noparens}
\newtheorem*{question*}{Question}
\newtheorem*{theorem*}{Theorem}

\newcommand{\uhr}{{\upharpoonright}}

\def\t{\tilde}

\def\mcal{\mathcal}
\def\mb{\mathbf}

\def\<{<^*}
\def\>{>^*}

\def\h{\hat}

\title{Which \DNR\ can be minimal}

\author{Lu Liu}
\address{Department of Mathematics,
Central South University,
City Changsha, Hunan Province,
China. 410083}
\email{g.jiayi.liu@gmail.com}

\begin{document}

\def\m{m}
\def\defeq{\overset{def}{=}}
\def\mbP{\mathbb{P}}
\def\U{U}
\def\Markov{Markov}
\def\mbE{\mathbb{E}}
\def\Tr{Tr}
\def\heighthomogeneous{height-homogeneous}
\def\allowsplit{allow split}
\def\DNR{\text{DNR}}

\maketitle

\begin{abstract}

In \cite{khan2017forcing}, Khan and Miller proved that
for every
computable non decreasing unbounded
function $h\in \omega^\omega$
(henceforth order function), if $h$ is sufficiently
large, then there exists
a $\DNR_h$ that is of minimal degree.
Where $h$ has to satisfy $\lim_{n\rightarrow\infty}
h(n)/(2^{k\cdot \prod_{m<n}h(m)})=\infty$ for all $k>0$.
Their core argument is that we can thin the
tree by a factor of $2^j$ to make
$j$ Turing functional split.
We improve their result by reducing this factor to $j$.
Thus we show that for every order function
$h$ with $\lim_{n\rightarrow\infty}
h(n)/( \prod_{m<n}h(m))^k=\infty$
for all $k>0$, there exists a $\DNR_h$ of minimal degree.
We answer a question of Brendle, Brooke-Taylor, Ng and Nies \cite{brendle2015analogy}
by showing that there exists a $G\in \omega^\omega$ such that
$G$ is weakly meager covering,
$G$ does not compute any Schnorr random real
and $G$ does not Schnorr cover REC.

\end{abstract}

\section{Introduction}

In \cite{khan2017forcing}, Khan and Miller proved that
for every sufficiently large
order function $h$, every oracle $X$, there exists
a $\DNR^X$ function, namely $G$, with $G\leq h$  that is of minimal degree.
Where $h$ has to satisfy $\lim_{n\rightarrow\infty}
h(n)/(2^{k\cdot \prod_{m<n}h(m)})=\infty$ for all $k$.
It is not known whether the fast growing condition on
$h$ can be eliminated. i.e., whether there exists,
for every order function $h$ and every oracle $X$, a $\DNR^X$ function $G$ that is of minimal
degree.
This question is closely related to another
question, whether there exists a real of hausdorff dimension
$1$ that is of minimal degree.
The existence of such real
 implies that the hausdorff dimension of reals of minimal degree
is $1$. A yes answer would confirm
the existence of such real.
We make a progress toward this direction
by improving Khan and Miller's result.
We release the
fast growing condition on $h$ to $\lim_{n\rightarrow\infty}
h(n)/( \prod_{m<n}h(m))^k=\infty$ for all $k>0$
(Theorem \ref{schth1}).

Cardinal characteristic study
has been an important direction in set theory.
The recent study of
Brendle, Brooke-Taylor, Ng and Nies\cite{brendle2015analogy}
pointed out an analog between
many results of cardinal characteristic and
results in computability theory.
\cite{brendle2015analogy} pointed out some analog between
notions in cardinal characteristic and computability theory
(mostly algorithmic randomness theory) and
shows how results in cardinal characteristic
can be translated in to results in computability theory.
Thus answering the computability theoretic analog
of a question cardinal characteristic provides ideas and directions for the original problem.
We answer a question in their paper
by showing that it is possible
 to avoid Schnorr randomness and Schnorr covering REC simultaneously
in \DNR\ (Theorem \ref{schth2})
where REC refers to the class of computable members
in $2^\omega$.
Khan and Miller \cite{khan2017forcing}, improving a theorem of Greenberg and Miller \cite{greenberg2011diagonally},
shows that there exists,
for every order function $h$,
a $\DNR_h$ function $G$ such that
$G$ does not compute any Kurtz random real
(and therefore does not compute any Schnorr random real).
The part of our proof concerning avoid Schnorr randomness,
takes a similar frame work as in \cite{khan2017forcing}
\cite{greenberg2011diagonally} but is different in combinatorial aspects.

Both of our results concerns bushy tree argument.
This argument is widely used in computability theory\cite{Kumabe2009fixed}\cite{Ambos-Spies2004Comparing}
\cite{greenberg2011diagonally}\cite{Beros2013DNC}\cite{Bienvenu2016Diagonally}\cite{Dorais2014Comparing}.
It's firstly invented by Kumabe (in an unpublished paper)
and later simplified by Kumabe and Lewis \cite{Kumabe2009fixed} to answer a question of
Sacks that whether there exists a $\DNR$ of minimal degree.
Ambos-Spies, Kjos-Hanssen, Lempp, and Slaman \cite{Ambos-Spies2004Comparing}
proved that over $\mathsf{RCA}$,
$\mathsf{WWKL}$ is stronger than
$\mathsf{DNR}$, answering a question of \cite{giusto2000located}.
A recent introduction of this method can be found in
\cite{khan2017forcing}.
See also remark \ref{schremark1} that how
bushy tree argument resembles many arguments
in reverse math.
We end up this section by introducing our notations and
the bushy tree argument.

\noindent \textbf{Notations}.
We write $(\Psi^\tau\uhr N)\downarrow$ if
$\Psi^\tau(n)\downarrow$ for all $n\leq N$.
We write $h^{<\omega}$ for
$\{\sigma\in \omega^{<\omega}:\sigma(n)\leq h(n)\text{ for all }n\leq |\sigma|\}$.
$\bot $ denote empty string.
For a tree $T$,
   we write $|\rho|_T$ for the $T$-length of $\rho$, i.e.~$|\rho|_T =n+1$ where $n$ is the number of proper initial segments of $\rho$ in $T$.
   For a string $\rho \in 2^{<\omega}$, we let $[\rho]^{\preceq}=\{\sigma: \sigma\succeq \rho\}$;
similarly, for a set $S\subseteq 2^{<\omega}$, let $[S]^{\preceq} = \{\sigma: \sigma\succeq \rho\text{ for some }\rho\in S\}$;
for a tree $T$, let $[T]$ denote the set of infinite paths on $T$;
for $\rho\in \omega^{<\omega}$,
 let $[\rho]= \{X\in \omega^\omega:X\succeq \rho\}$;
 for a finite set $V\subseteq 2^{<\omega}$, let $[V]=\{[\rho]:\rho\in V\}$.
For every non empty set $S\subseteq \omega^{<\omega}$,
let $\ell(S)$ denote the set of leaves of
$S$, i.e., $\{ \sigma\in S: [\sigma]^\preceq\cap S = \{\sigma\} \}$.
We define $\ell(\emptyset) = \{\bot\}$.
We say $\eta$ is the \emph{stem} of a tree
$T$ if $\eta\in T$ and $T\subseteq [\eta]^\preceq$.

\begin{definition}
For a function $p:\omega\rightarrow \mathbb{Q}$,
a tree $T$ is $p$-\emph{bushy }
over $\rho$
if
for every $\tau\in (T\cap [\rho]^\preceq)\cup \{\rho\}$ that is not a leaf
of $T$,
$\tau$ has
at least $p(|\tau|)$ many immediate successor in $T$.
$T$ is $p$-bushy from level $n$ to level $m$ if for every
$\tau\in T$ that is not a leaf, if $n\leq |\tau|\leq m$,
then $\tau$ has at least $p(|\tau|)$ many immediate successor in $T$.
$T$ is $p$-bushy above level $m$ if it is $p$-bushy from level
$m$ to level $\infty$.
A set $S$ is $p$-\emph{big} over $\rho$ if there
exists a finite tree $T$ that is $p$-bushy over $\rho$
such that $\ell(T)\subseteq S$;
$S$ is $p$-\emph{small} over $\rho$ if it is not
$p$-big over $\rho$.

\end{definition}
\begin{lemma}\label{schlem4}
Given two sets $B,C\subseteq \omega^{<\omega}$:
\begin{enumerate}

\item If $B\cup C$ is $(p+q)$-big over $\rho$,
then either $B$ is $p$-big over $\rho$
or $C$ is $q$-big over $\rho$.

\item If $B$ is $p$-big over $\rho$
and $C\subseteq B$ is $q$-small
over $\rho$,
then $B\setminus C$ is $(p-q)$-big over $\rho$.
\end{enumerate}
\end{lemma}
\begin{proof}
For item (1),
let $T$ be a finite tree witnessing $B\cup C$ is
$(p+q)$-big over $\rho$.
Let $\h{T} = \{\sigma\in T: B\text{ is }p\text{-big over }\sigma\}$.
If $\rho\in \h{T}$, then we are done since this means $B$ is $p$-big over $\rho$.
Suppose $\rho\notin \h{T}$. Note that for every $\sigma\in T\setminus \h{T}$
that is not a leaf of $T$,
$\sigma$ admits at least $q(|\sigma|)$ many immediate successors in $T$ that are
contained in $T\setminus \h{T}$
(otherwise it admits at least  $p(|\sigma|)$
many immediate successors in $T$ that are contained in $\h{T}$,
which implies $B$ is $p$-big over $\sigma$, a contradiction).
Therefore $\ell(T\setminus \h{T})\subseteq \ell(T)\subseteq B\cup C$.
But clearly $B\cap \ell(T)\subseteq \h{T}$. Therefore $\ell(T\setminus \h{T})\subseteq C$.
Since $T\setminus \h{T}$ is $q$-bushy over $\rho$,
$T\setminus \h{T}$ witnesses $C$ being $q$-big over $\rho$.
Item (2) follows from item (1) directly.
\end{proof}

\begin{remark}\label{schremark1}
Lemma \ref{schlem4} explains that how bushy tree method
resembles many arguments in reverse math.
A tree can be seen as an instance and its solution is
the infinite path through it. In order to make $G$
satisfy
multiple requirements say $\mcal{R}_0,\mcal{R}_1$,
we restrict $G$ on different trees $T_0,T_1$ where
$T_i$ forces $\mcal{R}_i$ and each of them is very bushy in the
sense that $\overline{T}_i$ is small. Note that
$T_0\cap T_1$ forces both requirements. It remains to
 show that $T_0\cap T_1$ is still combinatorially weak,
i.e., very bushy. This can be done by applying Lemma \ref{schlem4}
to show that
$\overline{T}_0\cup \overline{T}_1$ is still very small.
\end{remark}
\section{Generic Schnorr covering}
\label{schsec2}
 \def\mbV{\mb{V}}

 The main result in this section is
 Theorem \ref{schth2} which construct
 a $G\in\omega^\omega$ satisfying three
 weakness properties. First we introduce some necessary
 notions concerning Theorem \ref{schth2}.
Let REC denote the class of all computable members in $2^\omega$.
\begin{definition}[\cite{brendle2015analogy}]
A set $\mcal{A}\subseteq 2^\omega$ is
$A$-\emph{effectively meager} if there exists
a sequence of uniformly $\Pi_1^{0,A}$-classes
$(Q_m:m\in\omega)$ so that each $Q_m$ is nowhere dense
such that $\mcal{A}\subseteq \cup_m Q_m$.
A set $A$ is \emph{weakly meager covering}
if the class REC is $A$-effectively meager.
\end{definition}
For convenience, we adopt the following definition of
Schnorr test. Standard definition can be found in most text book
e.g. \cite{Nies2009Computability}\cite{Downey2010Algorithmic}.
\begin{definition}
 An \emph{$A$-schnorr test} is
 a sequence of finite sets $V_0,V_1,\cdots
 \subseteq 2^{<\omega}$ (denoted as
 $\mbV$) with the canonical index of $V_n$
 $A$-computable from $n$ so that
 $\m( V_n)\leq 2^{-n} $ for all $n$.
  We say $\mbV$ \emph{succeeds}
 on $X\in 2^\omega$ if
 $$X\in \bigcap\limits_{n\in\omega}
 \bigcup\limits_{m>n}[V_m].$$
A real $X$ is \emph{Schnorr random} if
 there is no Schnorr test
 succeeding on $X$.

 \end{definition}
 \begin{definition}[\cite{brendle2015analogy}]
We say $A$ \emph{Schnorr cover }
a set $\mcal{A}\subseteq 2^\omega$ iff there exists
an $A$-Schnorr test, namely $\mbV =(V_0,V_1,\cdots)$,
such that $\mbV$
succeeds on every $X\in \mcal{A}$.

 \end{definition}

Our main result in this section is the following which
answers question 4.1-(8) of \cite{brendle2015analogy}.
\begin{theorem}\label{schth2}
There exists a $G$ such that:
\begin{enumerate}
\item $G$ is weakly meager covering;
\item
$G$ does not compute any Schnorr random real;

\item $G$ does not Schnorr cover REC.
\end{enumerate}
\end{theorem}

The rest of this section will prove Theorem \ref{schth2}.
We firstly note that by results in
\cite{brendle2015analogy}\cite{Rupprecht2010}\cite{Kjos-Hanssen2011Kolmogorov},
 weakly meager covering is characterized  as following.
\begin{theorem}\label{schth0}
A set $X$ is weakly meager covering if and only if
it is high or of $\DNR$ degree.
\end{theorem}

By Theorem \ref{schth0}, to prove Theorem
\ref{schth2}
it suffices to construct a $G\in \DNR$ that satisfies
item (2)(3).
The proof follows the general steps as in
Mathias forcing (in computability theory),
or forcing on $\Pi_1^0$ class.
The condition is an effective closed set of some cantor space
which is seen as a collection of candidates of the $G$
we are constructing.
More specifically, it is a bushy tree who has very `few' leaves.
 We show how to extend a condition to
force a given requirement.
We deal with the requirements
of item (3), item (2) in section \ref{schsec0}
and section \ref{schsec1} respectively.

In this section, fix a computable monochromatically decreasing
positive function
$\varepsilon:\omega\rightarrow\mathbb{Q}$
  such that $\sum_{m\in\omega}\varepsilon(m)<1/4$.
\begin{definition}
For functions $q,p: \omega\rightarrow\mathbb{Q}$,
we say $(p,q)$ \emph{\allowsplit s}\
if for every $k\in\omega$,
$p(x)\varepsilon^k(x)>q(x)\geq 1 $ for all but finitely many $x$.
\end{definition}

A \emph{condition}
is a tuple $(\eta,T,p,q)$ such that
\begin{enumerate}
\item The tree $T$ is a computably bounded computable
tree (in $\omega^{<\omega}$) with $\eta$
as its stem;

\item  The functions $p,q$ are  computable function from
$\omega$ to $\mathbb{Q}$ such that $(p,q)$ \allowsplit\
and $p(x)>>q(x)$ for all $x\geq |\eta|$
(where $y>>x$ means $y$ is sufficiently larger
than $x$);

\item
The tree $T$ is $p$-bushy over $\eta$ and $\ell(T)$
is $q$-small over every $\rho$ such that $[\rho]\cap [T]\ne\emptyset$.
\end{enumerate}

As usual, a condition $(\eta,T,p,q)$ is seen as a collection,
namely $[T]$, of
the candidates of the $G$ we construct.
A condition $(\h{\eta},\h{T},\h{p},\h{q})$ \emph{extends} a condition
$(\eta,T,p,q)$
(written as $(\h{\eta},\h{T},\h{p},\h{q})\subseteq (\eta,T,p,q)$) if $[\h{T}]\subseteq [T]$.

An example of condition is the following.
Let $\t{p}:\omega\rightarrow\omega$ be sufficiently large
that $(\t{p},2)$ \allowsplit\ (where $2$ denote the constant function $2$),
let $\t{T}\subseteq (\t{p}+2)^{<\omega}$ be such a computable tree that
$[\t{T}] = \DNR_{\t{p}+2}$, $\t{T}$ is $\t{p}$-bushy over $\bot$
and $\ell(\t{T})$ is $2$-small over each
$\rho$ such that $[\rho]\cap [\t{T}]\ne\emptyset$ (see Lemma \ref{schlem3}).
Clearly $(\bot, \t{T},\t{p},2)$ is a condition. This will be our initial condition.

We need to satisfy the following two kinds
of requirements:
\begin{align}\nonumber
\mcal{R}_\Psi:&
\Psi^G\text{ as a Schnorr test  does not cover
REC};
\\ \nonumber
\mcal{R}'_{\Psi}:&
\Psi^G\text{ is not a Schnorr random real}.
\end{align}
A condition $(\eta,T,p,q)$ \emph{forces} a requirement
$\mcal{R}$ if every $X\in [T]$ satisfies $\mcal{R}$.
As usual, it suffices to show that every condition admit an
extension forcing a given requirement
since this enable us to construct a sequence of conditions
$$d_0=(\bot, \t{T},\t{p},2)\supseteq d_1\supseteq \cdots$$ so that every
requirement is forced by some $d_t$. Then
let $G\in \bigcap_t d_t$, which exists by compactness,
we have that $G$ satisfies all requirements.
Since $G\in d_0$, $G$ is a $\DNR$, thus we are done.

\subsection{Avoid Schnorr covering REC}\label{schsec0}
In this subsection we deal with requirement $\mcal{R}_\Psi$.
The final goal is Lemma \ref{schlem6} where the major
technique lies in Lemma \ref{schlem1}.

We will frequently use the following version of Markov inequality.
For a random variable $x$,
we write
$x\sim P$ to denote that $x$ follows the probability measure $P$;
we write $x|y\sim P$ to denote that
conditional on $y$, $x$ follows $P$.
For a finite set $S$, we use $U(S)$ to denote the uniform probability
measure on $S$.
\begin{lemma}\label{schlem12}
Let $S$ be a finite set
and let $f$ be a positive function on $S$.
If $\mbE_{x\sim \U(S)}[f(x)]<\lambda$,
then for every $\h{\lambda} >0$,
there exists a subset $S^*$ of $S$ such that
$|S^*|/|S|>
1-\lambda/\h{\lambda}$
and $f(x)<\h{\lambda}$
for all $x\in S^*$.
\end{lemma}

We frequently need to
pruned the tree so that for
some  $l\in\omega$, level $l$ is shrunk  into
a subset $S^*$ where the proportion of $S^*$
on that level is close to $1$.
To make sure that the tree is not severely
pruned, we work on exactly bushy tree
defined as following
(see Lemma \ref{schlem0}).

\begin{definition}
For a function $p:\omega\rightarrow \mathbb{Q}$,
a tree $T$ is \emph{exactly $p$-bushy }
over $\rho$
if
for every $\tau\in (T\cap [\rho]^\preceq)\cup \{\rho\}$ that is not a leaf
of $T$,
$\tau$ has
 $p(|\tau|)^+$ many immediate successor in $T$
 where $x^+$ denote the smallest integer $y$ such that
 $y\geq x$.
A tree $T$ is \emph{exactly $p$-bushy} from level $n$ to level $m$ if for every
$\tau\in T$ that is not a leaf, if $n\leq |\tau|\leq m$,
then $\tau$ has    $p(|\tau|)^+$ many immediate successor in $T$.
A tree $T$ is \emph{exactly $p$-bushy } above level $m$ if it is exactly $p$-bushy from level
$m$ to level $\infty$.
A finite set $S$ is \emph{exactly $p$-big} over $\rho$ if there
exists a finite tree $T$ that is exactly $p$-bushy over $\rho$
such that $\ell(T)= S$.

\end{definition}

\begin{lemma}\label{schlem0}
Let $\lambda,\varepsilon_0,\varepsilon_1,\cdots,\varepsilon_{n-1}>0$
satisfy $\lambda>\sum_{m< n}\varepsilon_m$.
Let $T\subseteq \omega^{\leq n}$ be a finite,
 exactly $p$-bushy (over empty string)
 tree with $\ell(T)\subseteq \omega^n$.
 Let $S\subseteq \ell(T)$
satisfies
$|S|/|\ell(T)|>\lambda$.
Then there exists a
%exactly
 $\h{p}$-bushy (over empty string) subtree
$\h{T}$ of $T$ such that
$\ell(\h{T})\subseteq S$ where
$\h{p}(m) = p(m)\varepsilon_m$ for all $m<n$.
\end{lemma}
\begin{proof}
We prove by induction on $n$. The conclusion holds trivially
for $n=1$.
Suppose it holds for $n-1$.
Since for every $\alpha\in T\cap \omega^1$,
$|\ell(T)\cap [\alpha]^\preceq|$ is identical,
therefore
$$\dfrac{|S|}{|\ell(T)|} =
\mathbb{E}_{\alpha\sim U(T\cap \omega^1)}
\big[\frac{|S\cap [\alpha]^\preceq|}
{|\ell(T)\cap [\alpha]^\preceq|}\big].$$
Let $$f(\alpha) = 1-\frac{|S\cap [\alpha]^\preceq|}
{|\ell(T)\cap [\alpha]^\preceq|},$$ we have
$$\mathbb{E}_{\alpha\sim \U(T\cap \omega^1)}[f(\alpha)]<1-\lambda.$$
Let $\h{\lambda} = \frac{1-\lambda}{1-\varepsilon_0}$
in the \Markov\ inequality \ref{schlem12}, we have that there exists a subset
$S_0$ of $T\cap \omega^{1}$ with
$$|S_0|>(1-(1-\lambda)\big/\frac{1-\lambda}{1-\varepsilon_0})|T\cap \omega^1|
=\varepsilon_0|T\cap \omega^1|
\geq \varepsilon_0p(0)$$ such that
for every $\alpha\in S_0$,
$$
\frac{|S\cap [\alpha]^\preceq|}
{|\ell(T)\cap [\alpha]^\preceq|}
=1-f(\alpha)>1- \frac{1-\lambda}{1-\varepsilon_0}> \lambda-\varepsilon_0.
$$
Thus, by induction (where $\lambda$
is substituted by $\lambda-\varepsilon_0$), there exists,
for each $\alpha\in S_0$,
a subtree $T_\alpha$ of $T$ with
$\alpha$ as its stem such that
$T_\alpha$ is  $\h{p}$-bushy over $\alpha$
and $\ell(T_\alpha)\subseteq S$.
Thus $\bigcup_{\alpha\in S_0} T_\alpha$
is the desired tree $\h{T}$.

\end{proof}

Before the core Lemma \ref{schlem1},  we introduce some terminology.
For a finite set $V\subseteq 2^{<\omega}$,
we write $\m(V)$ for $\m([V])$ where $\m$ is the Lebesgue
measure on $2^\omega$; and we write
$\m(V|V')$ for $\m([V]\cap [V'])/\m ([V'])$.
For every Turing functional $\Psi$,
every oracle $Y$, we assume
that $\Psi^Y$ is computing a Schnorr test,
namely $V_n = \Psi^Y(n)$
such that
$$\m(\cup_n \Psi^Y(n)) =\lambda^*$$
whenever $\Psi^Y$ is total
and $\lambda^* $ is sufficiently small.
Moreover, $$\m(\cup_{m\leq n}\Psi^Y(m))>\lambda^*-2^{-n-1}$$
if $\Psi^Y(m)\downarrow$ for all $m\leq n$.

%For convenience, we also assume
%$\Psi^Y(m)[t]\downarrow\rightarrow \Psi^Y(m')[t]\downarrow$
%for all $m'\leq m, t$ and
We write
\begin{align}
\nonumber
&\Psi^Y[t]\text{ for }\bigcup\limits_{m\leq t,\Psi^Y(m)[t]\downarrow}
\Psi^Y(m);
\\ \nonumber
&\Psi^Y(t_0,t_1]
\text{ for }
\Psi^Y[t_1]\setminus \Psi^Y[t_0].
\end{align}
Whenever we write $\Psi^\sigma[t]$, it implies
$|\sigma|>t$. i.e., for every
$\tau\succeq\sigma$, $\Psi^\tau[t] = \Psi^\sigma[t]$.

In the following text of this subsection,
let $\t{T}$ be a computably bounded
computable tree with
$\eta$ as its stem,
let $$T = \{\rho\in \t{T}:
[\rho]\cap [\t{T}]\ne\emptyset\}.$$
Suppose
$\t{T}$ is $p$-bushy over $\eta$ and
$\ell(\t{T})$ is $q$-small over every $\rho\in T$
where $p,q$ are
computable function.
The following Lemma is the core argument.

\begin{lemma}\label{schlem1}
 Suppose
 $q<p\varepsilon^3/8$ %-1$
 and for every $Y\in [\t{T}]$,
$\Psi^Y$ is total.
Then there exists a computable real $X\in 2^\omega$,
a computable tree $\h{T}\subseteq \t{T}$
with $\eta$ as its
stem   such that
\begin{enumerate}
\item $\h{T}$ is $p\varepsilon^3/8$-bushy
over $\eta$
and $\ell(\h{T})$ is $q$-small over
every $\rho$ such that $[\rho]\cap [\h{T}]\ne\emptyset$;

\item For every $Y\in [\h{T}]$,
$X\notin \bigcup_{m\in\omega}[\Psi^Y(m)]$.
\end{enumerate}

\end{lemma}
\begin{remark}
We will inductively define a sequence of trees $T_n$ and
a sequence of strings $\rho_n\in 2^{<\omega}$ so that
$$X=\lim\limits_{n\rightarrow\infty }\rho_n
\text{ and }
\{\rho\in \h{T}:[\rho]\cap [\h{T}]\ne\emptyset\}=
\bigcap\limits_{n\in\omega}\bigcup\limits_{m>n} T_m.$$
To this end, we maintain that
$$\m(\cup_{n'}\Psi^Y(n')|\rho_n)
\text{ is small for all }n
\text{ and }Y\in [\h{T}].
$$

Suppose we have found $\rho_n, T_{n+1}$
at time $t_{n+1}$ so that
for every $\sigma\in \ell(T_{n+1})$, $\m(\Psi^\sigma[t_{n+1}]|\rho_n)$
is small.
In order to find the next $\rho_{n+1}$,
wait for a time $t_{n+2}$ so that for some level $l_{n+2}$, for every
$\tau\in T[t_{n+2}]\cap \omega^{l_{n+2}}$,
$\m(\Psi^{\tau}[t_{n+2}])$ is sufficiently close to
$\lambda^*$, as to how close depends on
whatever constructed by step $n$, namely $T_{n+1},\rho_n$ etc.
To find the next $\rho_{n+1}\in 2^{m_{n+1}}$,
\begin{align}\nonumber
&\text{ we take an average of }\m(\Psi^{\tau}[t_{n+2}]|\rho)
\text{ over }\\ \nonumber
&\rho\in [\rho_n]^\preceq\cap 2^{m_{n+1}}
\text{ and }\tau\in T[t_{n+2}]\cap \omega^{l_{n+2}}.
\end{align}
We argue that
we can pruned the tree  so that  the average on that tree's leaves
(and a subset of $[\rho_n]^\preceq\cap 2^{m_{n+1}}$)
 can be much smaller than $\m(\Psi^\sigma[t_{n+1}]|\rho_n)$.
This is done by showing that for many $\rho$ and many $\tau$,
$\rho$ is not a member in $\Psi^{\tau}[t_{n+1}]$ since
$\m(\Psi^\sigma[t_{n+1}]|\rho_n)$ is small. Moreover, if $\rho$
is not a member in $\Psi^{\tau}[t_{n+1}]$,
many strings in the Schnorr test contributing to
$\m(\Psi^\sigma[t_{n+1}]|\rho_n)$
no longer contribute to $\m(\Psi^\sigma[t_{n+2}]|\rho)$.
i.e., $$\m(\Psi^{\tau}[t_{n+2}]|\rho) = \m(\Psi^{\tau}[t_{n+1},t_{n+2})|\rho).$$
Once we have proved that $\mathbb{E}_{\tau,\rho}[\m(\Psi^{\tau}[t_{n+1}]|\rho)]$
is small, by \Markov\ inequality \ref{schlem12}, there is a $\rho_{n+1}$ and
a large subset of $ T[t_{n+2}]\cap \omega^{l_{n+2}}\cap [\ell(T_n)]^\preceq$
so that $\m(\Psi^\tau[t_{n+2}]|\rho_{n+1})$
is sufficiently small for all $\tau$ in that subset.

In order to make sure that the pruned tree is sufficiently
bushy over \emph{sufficiently many }$\sigma\in \ell(T_{n+1})$
(so that the tree $T_{n+1}$ need not be pruned below certain level),
we will take the average over $\tau$
on an exactly bushy (above some level) tree. This is to avoid that the average is
mainly affected by a small portion of $\ell(T_{n+1})$ above which
the tree $T[t_{n+2}]\cap \omega^{l_{n+2}}\cap [\ell(T_{n+1})]^\preceq$
is much more bushy than the rest.
To ensure that for each level $l$,
the tree is not pruned below level $l$ after some point,
before we decide how to prune the tree,
we wait for a long enough time so that
$\Psi^\tau[t_{n+2}]$ is close enough to $\lambda^*$ for all
$\tau$ in level $l_{n+2}$ in $T[t_{n+2}]$,
as to how close depends on
whatever constructed by step $n$, namely $T_{n+1},\rho_n$ etc.

\end{remark}

\begin{proof}

 Since $ \t{T}\setminus T$ is $q$-small
 over each $\rho\in T$ with $q(x)<<p(x)$ for all
 $x\geq |\eta|$, By hypothesis on  $\t{T}$ and $T$,
for every $\sigma\in T$, every $m$,
there exists $t$ and a tree $T_\sigma$ with $\sigma$ as its stem
such that $T_\sigma$ is $p/2$-bushy over $\sigma$
and $\Psi^\tau(m)[t]\downarrow$ for all $\tau\in \ell(T_\sigma)$.

 \ \\

\noindent \textbf{Initial setup.}
Wait for such a  time $t_0$ that there exists a
tree $T_0\subseteq T[t_0]$,  such that:
\begin{enumerate}

\item For some $l_0\in\omega$, $T_0$ is $p/2$-bushy
from level $|\eta|$ to $l_0-1$ and $\ell(T_0)\subseteq \omega^{l_0}$.

\item for every $\tau\in\ell(T_0)$,
$\lambda_0<\m(\Psi^\tau[t_0])$.

\item Let $\overline{\lambda}_0= \lambda^*-\lambda_0$,
then $\overline{\lambda}_0$ is sufficiently small,
say $\lambda^*<\sqrt{\lambda_0}$.

\item Let $m_0$ be sufficiently large, say
 $\text{for every }
 \tau\in  T[t_{0}]\cap \omega^{l_{0}},
 \text{ every }\rho\in \Psi^{\tau}[t_{0}]$,
$ m_0 >|\rho|$.

\end{enumerate}

For a positive real $a$,
we write $a=o_s(A,B,\cdots)$ if $a$ is
very small, as to how small depends
on the object $A,B,\cdots$.

\noindent\textbf{Inductive hypothesis.}
Fix $n\geq -1$. Suppose by induction
 that we have computed
 \begin{itemize}
 \item a sequence of rationals $(\lambda_{\h{n}}: -1\leq \h n\leq n+1)$;
 \item a sequence of trees $(T_{\h{n}}:-1\leq \h n\leq  n+1)$ with
 $T_{\h{n}}\subseteq T[t_{\h n}]\cap \omega^{\leq l_{\h n}}$; and
 \item  a sequence of strings $(\rho_{\h n}\in 2^{m_{\h n}}: -1\leq \h n\leq n)$
 \end{itemize}
such that for every $-1\leq \h n\leq n$
\footnote{ In these  items, let
$ \lambda_{-1} = 0, l_{-2}= l_{-1} = |\eta|,  \rho_{-1} = \bot,
m_{-2} = m_{-1} =|\rho_{-1}|= 0.$}
\begin{enumerate}
\item Tree $T_{\h n+1}$
 is
 \begin{align}\nonumber
 & \frac{1}{2}p\varepsilon\text{-bushy from level }l_n\text{ to level }l_{\h n+1}-1,
 \\ \nonumber
&\frac{1}{4}p\varepsilon^2\text{-bushy
from level }l_{\h n-1}\text{ to level }l_{\h n}-1,\\ \nonumber
&\text{tree }T_{\h n+1}\cup (\t{T}\setminus T)
\text{ is }
\frac{1}{8}p\varepsilon^3\text{-bushy
from level }l_{-1}
\text{ to level }l_{\h n-1}-1;
\\ \nonumber
&\text{ moreover, }\ell(T_{\h n+1})\subseteq (\t{T}\setminus T)\cup \omega^{l_{\h n+1}}.
\end{align}
\item For every $\sigma\in T_{\h n+1}\cap \omega^{l_{\h n+1}}$,
\begin{align}\nonumber
&\lambda_0+\cdots+\lambda_{\h n+1}<m(\Psi^{\sigma}[t_{\h n+1}])
\text{ and }
\\ \nonumber
&\m(\Psi^\sigma[t_{\h n+1}]|\rho_{\h n})<\sqrt{\lambda_{\h n+1}}.
\end{align}

\item Let $\overline{\lambda}_{\h{n}+1}= \lambda^*-(\lambda_0+\cdots+\lambda_{\h{n}+1})$,
 we have
\begin{align}\nonumber
0<\overline{\lambda}_{\h{n}+1}
<o_s(\t{T}\cap \omega^{\leq l_{\h{n}}},m_{\h{n}},\overline{\lambda}_{\h{n}}).
%\min\big\{ \dfrac{1} {16^6|\t{T}\cap \omega^{\leq l_{\h{n}-1}}|^{6}},\dfrac{1}{(32\cdot 2^{m_{\h{n}-1}})^6}, %\frac{1}{2}\overline{\lambda}_{\h{n}-1} \big\}.
\end{align}

\item The integer $m_{\h n+1}>m_{\h n}$ is large enough so that
\begin{align}\nonumber
\text{ for every $\tau\in T[t_{\h n+1}]\cap \omega^{l_{n+1}},\rho\in \Psi^\tau[t_{\h n+1}]
$, $m_{\h n+1}>|\rho|$.
}
\end{align}
%  $ >m_{\h }$ $m_{\h n+1}>m_{\h n}+ \max\{|\rho|:\rho\in \Psi^\sigma[t_{\h n+1}], \}.$
\end{enumerate}

In addition, we also make the minor requirement that $t_{n+1}\leq l_{n+1}$,
which means $\Psi^\tau[t_{n+1}]=\Psi^\sigma[t_{n+1}]$
for all $\sigma\in T_{n+1}\cap \omega^{l_{n+1}}, \tau\in [\sigma]^\preceq$.
Intuitively, item (3) is because we choose $\lambda_{\h{n}}$
to be very close to $\overline{\lambda}_{\h{n}-1}$
as to how close depends
on whatever constructed by step $\h{n}-1$.
We refer these items as inductive hypothesis.
It's easy to verify that the inductive hypothesis
holds for $\lambda_0,  T_0$ and $\rho_{-1}$
(when $l<\h{l}$, $\h{T}$ being $\h{p}$-bushy from level
$\h{l}$ to level $l$ is meaningless and holds trivially)
especially checking that item (2) holds
since $$\m(\Psi^\sigma[t_0]|\rho_{-1})\leq \lambda^*<\sqrt{\lambda_0}$$
and item (3) holds since $|\t{T}\cap \omega^{\leq |\eta|}| = 1$
(and since $\lambda^*$ is sufficiently small).

\ \\

\noindent\textbf{Initial set up of step $n+1$.}
Now we construct the next tree $T_{n+2}$ and $\rho_{n+1}$.
As we said, we will firstly
wait for a time $t_{n+2}$ so that $\m(\Psi^\tau[t_{n+2}])$ to be close enough to
$\lambda^*$ for many $\tau\in T[t_{n+2}]$.
Then we
pruned the tree  (on whose leaves $\Psi^{\tau}[t_{n+2}]$
is close enough to $\lambda^*$) to an exact
bushy (above level $l_{n-1}$) tree.

By hypothesis of $\t{T}$,
there exists such a time $t_{n+2}$ and a $l_{n+2}\geq t_{n+2}$ such that
: for sufficiently many nodes $\tau$ in $\t T$, $m(\Psi^\tau[t_{n+2}])$
is sufficiently large. More precisely:
\begin{itemize}

 \item
 There exists
 a subset  $S$ of $T_{n+1}\cap \omega^{l_{n+1}}$
 such that for every $\alpha\in T_{n+1}\cap \omega^{l_{n-1}}$,
 either $\alpha\notin  T[t_{n+2}]$, or 
   the set of nodes between level $l_{n-1}$
 and $S$,
 namely $\{\alpha': \text{ for some }\sigma\in S, \alpha\preceq \alpha'\preceq\sigma\}$ is
 \begin{align}\nonumber
 &\text{ exactly }\frac{1}{8}p\varepsilon^2
 \text{-bushy from level
} l_{n-1}
\text{ to level }l_n-1,\\ \nonumber
&\text{ exactly }\frac{1}{4}p\varepsilon
\text{-bushy from level }
 l_n
 \text{ to level }l_{n+1}-1.
 \end{align}
Moreover, for every $\sigma\in S$, there exists
a subset  $S_\sigma $  of $[\sigma]^\preceq\cap T[t_{n+2}]\cap \omega^{l_{n+2}}$
 such that
 \begin{align}\nonumber
 S_\sigma\text{ is exactly }\frac{1}{2}p
 \text{-big over }\sigma.
\end{align}

 \item  For  every $\sigma\in S$, every $\tau\in S_\sigma$,
\begin{align}\nonumber
\lambda_0+\cdots+\lambda_{n+2}<m(\Psi^{\tau}[t_{n+2}]).
\end{align}

\item Where let $\overline{\lambda}_{n+2}=\lambda^*-(\lambda_0+\cdots+\lambda_{n+2})$,
we have
\begin{align}\label{scheq5}
0&<\overline{\lambda}_{n+2} <
o_s(\t{T}\cap \omega^{\leq l_{n+1}},m_{n+1},\overline{\lambda}_{n+1})
%\min\big\{\dfrac{1}
%{16^6|\t{T}\cap \omega^{\leq l_{n+1}}|^{6}},
%\dfrac{1}{(32\cdot 2^{m_{n+1}})^6}, \frac{1}{2}\overline{\lambda}_{n+1} \big\}.
\end{align}

\item
Let $m_{n+2} >m_{n+1}$ be large enough so that,
\begin{align}\label{scheq8}
\text{ for every }\tau\in  T[t_{n+2}]\cap \omega^{l_{n+2}},
\text{ every }\rho\in \Psi^{\tau}[t_{n+2}],
 m_{n+2} >|\rho|.
\end{align}
\end{itemize}

Note that such set $S$ does exist since
$T_{n+1}$ is $p\varepsilon^2/4$-bushy
from level $l_{n-1}$ to level $l_n-1$,
$p\varepsilon/2$-bushy from level $l_n$ to level $l_{n+1}-1$
 and $T_{n+1}\setminus   T[t_{n+2}]\subseteq
\t{T}\setminus T$ is $q$-small over each $\rho\in T$ with
$q<p\varepsilon^3/8$.

\def\oc{1^{-}}
\def\Oc{0^+}

Now we construct $\rho_{n+1}$
and shrink the set $T[t_{n+2}]\cap \omega^{l_{n+2}}$
so that for sufficiently many $\tau$ in the shrinked set,
we have
\begin{align}\nonumber
m(\Psi^\tau[t_{n+2}]|\rho_{n+1})<\sqrt{\lambda_{n+2}}.
\end{align}
To this end, we
shrink $S$ to $S^*$ so that
for every $\sigma\in S^*$,
for many $\rho\in 2^{m_{n+1}}\cap [\rho_n]^\preceq$,
$\rho$ is not a member
of $\Psi^\sigma[t_{n+1}]$.
And
 for every $\sigma\in S^*$, we shrink
$S_\sigma$ to $S^*_\sigma$ so that
for every $\tau\in S_{\sigma}^*$,
$\m(\Psi^\tau[t_{n+2}]|\rho_{n+1})$ is small.
The key point is that shrinking
$S$ to $S^*$ does not thin out any nodes
below level $l_{n-1}$. The tree $T_{n+2}$
will be constructed according to $S^*$
and $(S^*_\sigma: \sigma\in S^*)$.
 See Figure \ref{schfig1}
and the explanatory note to have an intuition
of how much tree is pruned.
\begin{claim}\label{schclaim0}
There exists a tree $T_{n+2}$ with $\eta$
as its stem and a $\rho_{n+1}\in [\rho_n]^\preceq\cap 2^{m_{n+1}}$
such that:
\begin{enumerate}
\item Below level $l_{n-1}$, nothing is thinned out.
i.e., $$T_{n+2}\cap \omega^{\leq l_{n-1}} = T_{n+1}\cap \omega^{\leq l_{n-1}}
\cap T[t_{n+2}].$$
Above level $l_{n-1}$,
all leaves lies in level $l_{n+2}$. i.e.,
 $$\ell(T_{n+2})\cap \omega^{\geq l_{n-1}}\subseteq \omega^{l_{n+2}}.$$

\item The tree $T_{n+2}\subseteq T[t_{n+2}]\cap \omega^{\leq l_{n+2}}$
 is
 \begin{align}\nonumber
 \hspace{2cm}&\frac{1}{2}p\varepsilon\text{-bushy from level }
l_{n+1}\text{ to level }l_{n+2}-1,\\ \nonumber
& \frac{1}{4}p\varepsilon^2\text{-bushy
from level }l_{n}
\text{ to level }l_{n+1}-1\\ \nonumber
&\frac{1}{8}p\varepsilon^3\text{-bushy
from level }l_{n-1}
\text{ to level }l_n-1
\text{ and };
\\ \nonumber
&\text{ the tree }T_{n+2}\cup (\t{T}\setminus T) \text{ is }
\frac{1}{8}p\varepsilon^3
\text{-bushy
from level }0\text{ to level }l_{n-1}-1;
\end{align}
\item For every $\tau\in T_{n+2}\cap \omega^{l_{n+2}}$,
\begin{align}\nonumber
\m(\Psi^\tau[t_{n+2}]|\rho_{n+1})<\sqrt{\lambda_{n+2}}.
\end{align}

\item Moreover, the inductive hypothesis holds.

\end{enumerate}
\end{claim}

\begin{figure}
\centering
\includegraphics[width=1.0\textwidth]{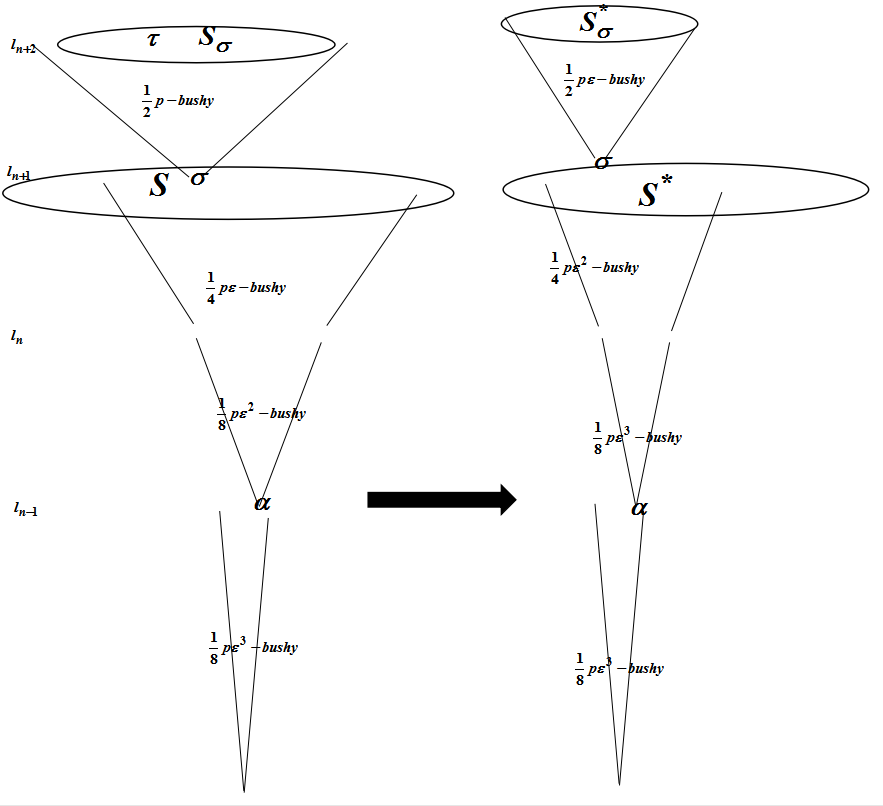}
\caption{Intuitively, the tree
is thinned by a factor of $\varepsilon$ above level $l_{n-1}$
so  that for every  $\tau\in T_{n+2}\cap \omega^{l_{n+2}}$,
$\m(\Psi^\tau[t_{n+2}]|\rho_{n+1})<\sqrt{\lambda_{n+2}}$.
Meanwhile, below level $l_{n-1}$, nothing is thinned out.
}
\label{schfig1}
\end{figure}
\begin{proof}
Item (3)(4) of the inductive assumption
are verified by definition of
$\lambda_{n+2}$ (\ref{scheq5}) and
$m_{n+2}$ (\ref{scheq8}).
Item (1)(2) of the inductive assumption will be verified
by item (2)(3) of this Claim. Therefore it remains to deal with
item (1)(2)(3) of this Claim.
In the following text, we write $\lambda^{\oc}$ if it is
$\lambda^c$ for some constant $c$ that can be chosen arbitrarily close
to $1$ and write $\lambda^{\Oc}$ if it is
$\lambda^c$ for some constant $c$  bounded away from
$0$.

For every $\rho\in 2^{m_{n+1}}$,
 let
 $$A_\rho = \{\sigma\in S: \rho\notin [\Psi^{\sigma}[t_{n+1}]]^\preceq\}.$$
Since $m_{n+1}$ is sufficiently large
(by inductive hypothesis (4)), therefore:
\begin{align}\nonumber
\text{ for every }\rho\in 2^{m_{n+1}},
\text{ every }\sigma\in S,
\text{ either }
\rho\in [\Psi^\sigma[t_{n+1}]]^\preceq
\text{ or }
[\rho]\cap [\Psi^\sigma[t_{n+1}]]=\emptyset.
\end{align}
Thus we can rewrite the inductive assumption (2), namely
  \begin{align}\nonumber
&  \text{ “for every }\sigma\in S,
  \m(\Psi^\sigma[t_{n+1}]|\rho_n)<\sqrt{\lambda_{n+1}}
  \text{ ”
  as:}
\\ \label{scheq19}
& \mbP_{(\sigma,\rho)\sim \U(S\times ([\rho_n]^\preceq
 \cap 2^{m_{n+1}}))}
 \bigg(
 \sigma\notin A_\rho
 \bigg)<\sqrt{\lambda_{n+1}}.
 \end{align}

Take the function $f(\rho) =
 \mbP_{\sigma\sim \U(S)}
 \bigg(
 \sigma\notin A_\rho
 \bigg)$ and rewrite (\ref{scheq19}) as
  $$\mathbb{E}_{\rho\sim
 \U([\rho_n]^\preceq
 \cap 2^{m_{n+1}})}[f(\rho)]<\sqrt{\lambda_{n+1}}.$$
 Let $\h{\lambda}= \lambda_{n+1}^{\Oc}$ in the \Markov\ inequality \ref{schlem12},
  we have
 there exists
 \begin{align}\nonumber
 &\text{ a
  subset } R
  \text{  of }
  [\rho_n]^\preceq\cap 2^{m_{n+1}}
  \text{  with }
  \dfrac{|R|}{|[\rho_n]^\preceq\cap 2^{m_{n+1}}|}>1-\lambda_{n+1}^{\Oc}
\\ \nonumber
&\text{ such that }f(\rho)<\lambda_{n+1}^{\Oc}
\text{ for all }\rho\in R.
\end{align}
By definition of $f$,  $f(\rho)<\lambda_{n+1}^{\Oc}$
translates to $ |A_\rho|/|S|>1-\lambda_{n+1}^{\Oc}$.
In summary:

\begin{align}\label{scheq0}
\dfrac{|R|}{|[\rho_n]^\preceq\cap 2^{m_{n+1}}|}=\dfrac{|R|}{2^{m_{n+1}-m_n}}&>1-\lambda_{n+1}^{\Oc} % {\frac{1}{4}}
\\ \nonumber
\text{ and for every }
\rho\in R,\ \ \ \  \dfrac{|A_\rho|}{|S|}&>1-\lambda_{n+1}^{\Oc} %{\frac{1}{4}}.
\end{align}

The key note is the following, which says that
 for most $\rho\in 2^{m_{n+1}}\cap [\rho_n]^\preceq$,
 $\m(\Psi^\tau[t_{n+2}]|\rho) $ is much smaller
 than $\m(\Psi^\tau[t_{n+2}]|\rho_n)$
  as long as
 $\rho$ is not a member of $\Psi^\tau[t_{n+1}]$ since
 in that case  $$\m(\Psi^\tau[t_{n+2}]|\rho)
 =\m(\Psi^\tau(t_{n+1},t_{n+2}]|\rho), $$
 which is, in average,
  small since
  $$\m(\Psi^\tau(t_{n+1},t_{n+2}])<\overline{\lambda}_{n+1}
<o_s(m_n).$$
 More precisely,
for every $\sigma\in S, \tau\in S_\sigma$
\footnote{Recall that for every $\sigma\in S$, since $|\sigma|=l_{n+1}\geq t_{n+1}$,
so $\Psi^\tau[t_{n+1}]=\Psi^\sigma[t_{n+1}]$
for all $\tau\succeq\sigma$.}:
\begin{align}\label{scheq1}
&\sum\limits_{\rho\in
R\setminus [\Psi^{\sigma}[t_{n+1}]]^\preceq}
\m(\Psi^\tau[t_{n+2}]|\rho)\cdot\m(\rho|\rho_n)
\\ \nonumber
= &
\sum\limits_{\rho\in
R\setminus [\Psi^{\sigma}[t_{n+1}]]^\preceq}
\m(\Psi^\tau(t_{n+1},t_{n+2}]|\rho)\cdot \m(\rho|\rho_n)
\\ \nonumber
\leq &\sum\limits_{\rho\in
[\rho_n]^\preceq\cap 2^{m_{n+1}}}
\m(\Psi^\tau(t_{n+1},t_{n+2}]|\rho)\cdot\m(\rho|\rho_n)
\\ \nonumber
=&\ \ \m(\Psi^\tau(t_{n+1},t_{n+2}]|\rho_n)
\\ \nonumber
\leq&\ \ 2^{m_n}\overline{\lambda}_{n+1}
\hspace{1.5cm}\text{since }
\m(\Psi^\tau(t_{n+1},t_{n+2}])<\overline{\lambda}_{n+1}
\\ \nonumber
<&%\frac{1}{4}
\ \ \overline{\lambda}_{n+1}^{\oc}%{\frac{5}{6}}
\hspace{2cm}\text{since }\overline{\lambda}_{n+1}<o_s(m_n).
\end{align}

Now we prove that $|R\cap [\Psi^{\sigma}[t_{n+1}]]^\preceq|$
is very small compared to $|R|$ for all $\sigma\in S$:
\begin{align}\label{scheq7}
|R\cap [\Psi^{\sigma}[t_{n+1}]]^\preceq|
&= 2^{m_{n+1}}\cdot \m(R\cap [\Psi^{\sigma}[t_{n+1}]]^\preceq)
\\ \nonumber
&\leq 2^{m_{n+1}}\cdot\m( [\rho_n]^\preceq\cap [\Psi^{\sigma}[t_{n+1}]]^\preceq)
\hspace{1cm}\text{ since }R\subseteq [\rho_n]^\preceq
\\ \nonumber
&=2^{m_{n+1}-m_n}\cdot\m(\Psi^{\sigma}[t_{n+1}]|\rho_n)
\\ \nonumber
&<2^{m_{n+1}-m_n}\cdot\sqrt{\lambda_{n+1}}
\hspace{3cm}\text{by inductive hypothesis (2)}.
\end{align}
Therefore by (\ref{scheq0})(\ref{scheq7}),
$R\setminus [\Psi^\sigma[t_{n+1}]]^\preceq$ is very large in $[\rho_n]^\preceq\cap 2^{m_{n+1}}$:
\begin{align}
|R\setminus [\Psi^\sigma[t_{n+1}]]^\preceq|
&= |R|- |R\cap [\Psi^\sigma[t_{n+1}]]^\preceq|
\\ \nonumber
&\geq 2^{m_{n+1}-m_n}
(1-\lambda_{n+1}^{\Oc})
%{\frac{1}{4}}
%-\sqrt{\lambda_{n+1}})
%\\ \nonumber
%>2\lambda_{n+1}^{\frac{1}{4}})
\end{align}
Thus,
 \begin{align}\label{scheq12}
 \frac{1}{|R\setminus [\Psi^\sigma[t_{n+1}]]^\preceq|}
%<\frac{1+\lambda_{n+1}^{\oc}
%3\lambda_{n+1}^{\frac{1}{4}}
%}{2^{m_{n+1}-m_n}}
<
\frac{2}{2^{m_{n+1}-m_n}}=2\m(\rho|\rho_n).
\end{align}
For $\sigma\in S$, $\rho\in R$, let
$$
C(\sigma|\rho) =
\mbE_{\tau\sim \U(S_\sigma)}[\m(\Psi^\tau[t_{n+2}]|\rho)].
$$
By (\ref{scheq12}), replace $\m(\rho|\rho_n)$
in (\ref{scheq1}) by $1/|R\setminus [\Psi^\sigma[t_{n+1}]]^\preceq|$,
we have:
for every $\sigma\in S$,
\begin{align}\label{scheq13}
\mbE_{\rho\sim \U(R\setminus [\Psi^\sigma[t_{n+1}]]^\preceq)}
[C(\sigma|\rho)]<
%\frac{1}{2}
\overline{\lambda}_{n+1}^{\oc}%{\frac{5}{6}}
.
\end{align}
Next, we show that there exists a $\rho_{n+1}\in R$
such that,
 $$\mbE_{\sigma\sim\U(A_{\rho_{n+1}})}[C(\sigma|\rho_{n+1})]
\approx \mbE_{\rho\sim \U(R\setminus [\Psi^{\h\sigma}[t_{n+1}]]^\preceq)}
[C(\h\sigma|\rho)].$$
Averaging (\ref{scheq13})
 over $\sigma$:

\begin{align}\label{scheq2}
\mbE_{(\sigma,\rho)\sim P}
\big[
C(\sigma|\rho)
\big]<
%\frac{1}{2}
\overline{\lambda}_{n+1}^{\oc}%{\frac{5}{6}}
.
\end{align}

Where $P$ is such a probability measure
on $S\times R$ that if $(\sigma,\rho)\sim P$,
then $\sigma\sim \U(S)$ and
$\rho|\sigma\sim \U(R\setminus [\Psi^\sigma[t_{n+1}]]^\preceq)$.
Consider another probability measure
 $P'$ on $S\times R$ that if $(\sigma,\rho)\sim P'$,
 then $\rho\sim \U(R)$
 and $\sigma|\rho\sim \U(A_\rho)$.
 Note that the support of $P$
 and $P'$ (denoted as $supp(P'),supp(P)$ respectively) are identical
 (namely $\{(\sigma,\rho)\in S\times R:\rho\notin [\Psi^\sigma[t_{n+1}]]^\preceq \}$) and
  for every $(\sigma,\rho)\in supp(P)$:
 \begin{align}\nonumber
 &P(\sigma,\rho)\geq \frac{1}{|S|\cdot |R|}
 \text{ and by (\ref{scheq0}) }
 P'(\sigma,\rho) = \frac{1}{|R|\cdot |A_\rho|}
 <\frac{2}{|S|\cdot |R|},
 \\ \nonumber
 &\text{which means }
 \frac{P'}{P}(\sigma,\rho)<2
 \text{ for all }(
 \sigma,\rho)\in supp(P).
 \end{align}
 Therefore we can rewrite
 (\ref{scheq2}) as
 \begin{align}\nonumber%\label{scheq3}
\mbE_{(\sigma,\rho)\sim P'}
\big[
C(\sigma|\rho)
\big]<
\overline{\lambda}_{n+1}^{\oc}.%{\frac{5}{6}}.
\end{align}
This implies, by definition of $P'$, that there exists a member in
$ R$, namely $\rho_{n+1}$,
such that
 \begin{align}\nonumber%\label{scheq3}
\mbE_{\sigma\sim \U(A_{\rho_{n+1}})}
\big[
C(\sigma|\rho_{n+1})
\big]<\overline{\lambda}_{n+1}^{\oc}.
%{\frac{5}{6}}.
\end{align}
Applying \Markov\ inequality \ref{schlem12},
there exists
an $S^*\subseteq A_{\rho_{n+1}}$
such that
\begin{align}\label{scheq4}
&\dfrac{|S^*|}{|A_{\rho_{n+1}}|}>
1-\overline{\lambda}_{n+1}^{\Oc}
%{\frac{1}{6}};
\text{ and for every }
\sigma\in S^*\\ \nonumber
&C(\sigma|\rho_{n+1})<\overline{\lambda}_{n+1}^{\oc}.
%{\frac{4}{6}}.
\end{align}
Unfolding the definition of $C(\sigma|\rho_{n+1})$
and applying \Markov\ inequality \ref{schlem12},
for every $\sigma\in S^*$,
there exists an $S^*_\sigma\subseteq S_\sigma$
such that
\begin{align}\label{scheq10}
&\dfrac{|S^*_\sigma|}{|S_\sigma|}>1-\overline{\lambda}_{n+1}^{\Oc}%{\frac{1}{6}}
\text{ and for every }\tau\in S^*_\sigma\\ \nonumber
&\m(\Psi^\tau[t_{n+2}]|\rho_{n+1})<
\overline{\lambda}_{n+1}^{\oc}%{\frac{1}{2}}
<\sqrt{\lambda_{n+2}}.
\end{align}

We will ensure that $\ell(T_{n+2})\cap \omega^{l_{n+2}}\subseteq \bigcup_{\sigma\in S^*}S_\sigma^*$,
therefore
(\ref{scheq10}) verifies item (3) of the Claim.
Now we construct $T_{n+2}$ according to $S^*$
and $(S^*_\sigma: \sigma\in S^*)$.
Recall that the key point is that below level
$l_{n-1}$, nothing has to be thinned out.
This is because the proportion of
$S\setminus S^*$ over $S$ (namely
$\lambda_{n+1}^{\Oc}$) is very small
compared to $1/|\t{T}\cap \omega^{\leq l_{n-1}}|$.
By (\ref{scheq4})(\ref{scheq0}),
$$
\frac{|S^*|}{|S|}>1-%2
\lambda_{n+1}^{\Oc}.%{\frac{1}{6}}.
$$
Since $S$ is exactly big above level $l_{n-1}$, for every $\alpha\in T_{n+1}\cap \omega^{l_{n-1}}\cap T[t_{n+2}]$,
$|S\cap [\alpha]^\preceq|$ is identical, therefore
$|S\cap [\alpha]^\preceq| = \frac{|S|}{|T_{n+1}\cap \omega^{l_{n-1}}\cap T[t_{n+2}]|}$.
Thus:
\begin{align}
\dfrac{|S^*\cap [\alpha]^\preceq|}{|S\cap [\alpha]^\preceq|}
&\geq 1-\dfrac{|S\setminus S^*|}{|S\cap [\alpha]^\preceq|}\\ \nonumber
&=1-\dfrac{|S\setminus S^*|}{|S|}\cdot
|T_{n+1}\cap \omega^{l_{n-1}}\cap T[t_{n+2}]|\\ \nonumber
&>1-
%4
\lambda_{n+1}^{\Oc}%{\frac{1}{6}}
\cdot |\t{T}\cap \omega^{\leq l_{n-1}}|\\ \nonumber
&>1-\overline{\lambda}_n^{\Oc}%{\frac{1}{6}}
\cdot |\t{T}\cap \omega^{\leq l_{n-1}}|
\\ \nonumber
&>\frac{3}{4}
\ \ \ \text{since }\overline{\lambda}_n<
o_s(\t{T}\cap \omega^{\leq l_{n-1}}).
\end{align}

Since $S$ is  exactly
$p\varepsilon^2/8$-big from level
 $l_{n-1}$ to level $l_n-1$, and
exactly $p\varepsilon/4$-big from level
 $l_n$ to level $l_{n+1}-1$,
 by Lemma \ref{schlem0},
 for every $\alpha\in T_{n+1}\cap \omega^{l_{n-1}}\cap T[t_{n+2}]$,
 there exists a finite tree $T^*_\alpha \subseteq T_{n+2}$ with $\alpha$ as its stem
 such that $\ell(T^*_\alpha)\subseteq S^*$
 and $T^*_\alpha$ is
 \begin{align}\nonumber
  &  \frac{1}{8}p\varepsilon^3\text{-bushy from level }
 l_{n-1}\text{ to level }l_n-1;\\ \nonumber
 & \frac{1}{4}p\varepsilon^2\text{-bushy from level
 }l_n
 \text{ to level }
 l_{n+1}-1.
 \end{align}

Since $S_\sigma$ is
exactly $p/2$-big over $\sigma$,
by Lemma \ref{schlem0}
and (\ref{scheq10}), for every $\sigma\in S^*$,
there exists a finite tree
$T_\sigma^*$ with $\sigma$ as its stem such that   $\ell(T_\sigma^*)\subseteq S_\sigma^*$
and $T_\sigma^*$ is
\begin{align}\nonumber
\text{ $p\varepsilon/2$-bushy  over $\sigma$.}
\end{align}
 Let
 \begin{align}\nonumber
 &T^* = \bigcup\{ T^*_\alpha: \alpha\in T_{n+1}\cap \omega^{l_{n-1}}\cap T[t_{n+2}]\}%\limits_{\alpha\in T_{n+1}\cap \omega^{l_{n-1}}\capT[t_{n+2}]}
 \text{ and
 }\\ \nonumber
 &T_{n+2} =%\defeq
 ( T_{n+1}\cap \omega^{\leq l_{n-1}}\cap T[t_{n+2}])
 \cup T^*\cup \bigcup\limits_{\sigma\in \ell(T^*)}T^*_\sigma.
 \end{align}
Combine
with (\ref{scheq10}), $T_{n+2}, \rho_{n+1}$ are the desired
 tree and string. Thus we are done.

\end{proof}

Let
$$\h{T} = \bigcup\limits_{n\geq 0}
(T_{n+2}\cap \omega^{\geq l_{n-1} }\cap \omega^{\leq l_{n}}).
$$

Note that  for every
$\rho$ in the tree
$$ \h{T}\cap \omega^{\geq l_{n-1} }\cap \omega^{\leq l_{n}}
= T_{n+2}\cap \omega^{\geq l_{n-1} }\cap \omega^{\leq l_{n}},$$
the only reason for $[\rho]\cap [\h{T}]=\emptyset$ is that
$\rho$ enters $\overline{T}$.
Therefore, for every $\rho\in \h{T}\cap T$,
$$\h{T}\setminus T\subseteq \t{T}\setminus T
\text{ is }q
\text{-small over }\rho.
$$
Therefore,
\begin{align}\nonumber
 &\text{ every }
 \rho\in \h{T}\cap T
 \text{  has at least }
\\ \nonumber
& \frac{1}{8}p(|\rho|)\varepsilon(|\rho|)^3-q(|\rho|)
\text{ many immediate successors in }\h{T}\cap T.
\end{align}
Thus $$\{\rho\in \h{T}:[\rho]\cap [\h{T}]\ne\emptyset\} = \h{T}\cap T.$$
By definition of $T_n$,
$\h{T}$ is $p\varepsilon^3/8$-bushy over $\eta$
(see item (2) of Claim \ref{schclaim0}).
Let $X= \cup_n \rho_n$.
If for some $Y\in [\h{T}]$,
some $\rho\in \Psi^Y(m)$,
$X\in [\rho]$, then there must exists
$n$ such that $$\m(\Psi^Y[t_{n+1}]|\rho_n)=1>\sqrt{\lambda_{n+1}},$$
a contradiction
to item (3) of Claim \ref{schclaim0}.
Thus $\h{T},X$ is the desired tuple
as in this Lemma.

\end{proof}

Lemma \ref{schlem1} requires that for every
$Y\in [\t{T}]$, $\Psi^Y$ is total.
Starting with
the condition $(\eta,\t{T},p,q)$,
the following Lemma \ref{schlem2}
shows that we can easily shrink the tree $\t{T}$
so that either every $Y$ on the remaining
subtree makes $\Psi^Y$ total, or
every $Y$ on the remaining subtree makes
$\Psi^Y$ non total.
Moreover, the remining subtree is still sufficiently
bushy, say $\h{p}$-bushy over its stem where
$\h{p}$ is a given computable function  that is
very small compared to $p$, but is very large
compared to $q$.
i.e., both $(p,\h{p})$ and $(\h{p},q)$ \allowsplit.

\begin{lemma}\label{schlem2}
Suppose for every $\rho\in T$,
every $m$,
$$\text{ the set }
\big\{\tau\in \t{T} :
\Psi^\tau(m)\downarrow\big\}
\text{ is }\h{p}\text{-big over }\rho.
$$
Then there exists a computable
subtree $\h{T}$ of $\t{T}$ with
$\eta$ being its stem such that:
\begin{enumerate}
\item $\h{T}$ is $\h{p}$-bushy over $\eta$;
\item $\ell(\h{T})$ is $q$-small over each $\rho$
such that $[\rho]\cap [\h{T}]\ne\emptyset$;
\item For every $Y\in [\h{T}]$,
$\Psi^Y$ is total.
\end{enumerate}
\end{lemma}

\begin{proof}

We compute a sequence of finite trees $\h{T}_s,s\in\omega$
as following.
Let $\h{T}_0 = \{\eta\}$ and let $\h{T}_{-1} = \emptyset$
for convenience.
Suppose we have computed $\h{T}_s$.
To compute $\h{T}_{s+1}$: wait for such a time $t$ that
for every $\sigma\in \ell(\h{T}_{s}\setminus \h{T}_{s-1})$,
either $\sigma\notin T[t]$ or there exists
a  tree $T_\sigma\subseteq \t{T}$ with $\sigma$ as its stem such that
 $T_\sigma$ is $\h{p}$-bushy over $\sigma$ and
 $\Psi^\tau(s+1)\downarrow$ for all
 $\tau\in\ell(T_\sigma)$.
 Let
 $$
 \h{T}_{s+1} = \h{T}_s\cup \bigcup\{T_\sigma: \sigma\in\ell(\h{T}_{s}\setminus \h{T}_{s-1})
 \cap T[t]\}. %\limits_{\sigma\in\ell(\h{T}_{s}\setminus \h{T}_{s-1})  \cap T[t]}T_\sigma.
 $$
By hypothesis of the lemma,
such $t$ must exist.
Clearly $\h{T} = \cup_s \h{T}_s\subseteq \t{T}$
is $\h{p}$-bushy over $\eta$.
Moreover, $\ell(\h{T})\subseteq \t{T}\setminus
T$. Therefore
$\ell(\h T)$ is $q$-small over each $\rho\in T$.
Meanwhile, for every $\rho\in \h{T}$ with
$[\rho]\cap [\h{T}]\ne\emptyset$,
we have $\rho\in T$.
Thus item (2) of this Lemma is verified.
It is trivial to verify that
$\Psi^Y$ is total for all $Y\in [\h{T}]$. Thus we are done.

\end{proof}

It's convenient to note the following.
\begin{lemma}\label{schlem3}
Suppose $\t{T}$ is a computable tree with $\eta$
as its stem
and $T\subseteq \t{T}$ is a
pruned co.c.e. tree with $\eta$
as its stem such that
$\t{T}\setminus T$ is $q$-small over each $\rho\in T$
and $\t{T}$ is $p$-bushy over $\eta$.
Then there exists a computable tree $\h{T}\subseteq \t{T}$
with $\eta$ as its stem
such that $[\h{T}]= [T]$, $\h{T}$ is $p$-bushy over $\eta$ and
$\ell(\h{T})$ is
$q$-small over each $\rho$ such that $[\rho]\cap[\h{T}]\ne\emptyset$.
\end{lemma}

Now we are ready to finally prove that:

\begin{lemma}\label{schlem6}
Every condition
admits an extension that forces
$\mcal{R}_\Psi$.

\end{lemma}
\begin{proof}
 Fix a condition $(\eta,\t{T},p,q)$
 and recall that $$T = \big\{
\rho\in \t{T}:[\rho]\cap [\t{T}]\ne\emptyset
\big\}.
$$
Clearly $T$  is a co-c.e. pruned
subtree of $\t{T}$
(and with $\eta$ as its stem), $T$ is $(p-q)$-bushy
over $\eta$ and $\t{T}\setminus T$ is
$q$-small over every $\rho\in T$.
 Suppose $|\eta|$ is sufficiently
large so that there exists a computable function
$\h{p}:\omega\rightarrow\mathbb{Q}$
such that $(p,2\h{p}), (\h{p},q)$ \allowsplit\
and $p(x)>>\h{p}(x)>>q(x)>>x$ for all $x\geq |\eta|$.

\ \\

\textbf{Case 1.}
There exist a $\xi\in T$,
$m\in\omega$
such that
$\big\{\tau\in \t{T}:
\Psi^\tau(m)\downarrow\big\}$
is not $\h{p}$-big over $\xi$.

Let $$\h{T} =
\big\{
\tau\in T\cap [\xi]^\preceq:
\{\tau'\in \t{T}:
\Psi^{\tau'}(m)\downarrow\}
\text{ is not }\h{p}\text{-big over }\tau
\big\}.$$
Clearly $\h{T}$ is a co-c.e.  tree
with $\xi$ as its stem. By bushy tree combinatorics
and since $\h{p}+q<p$,
$\h{T}$ is pruned.
Actually every $\rho\in\h{T}$
has at least $p(|\rho|)-\h{p}(|\rho|)-q(|\rho|)$
many immediate successor in $\h{T}$
since $T$ is $(p-q)$-bushy over $\xi$.
Note that  by definition of $\h{T}$,
$$\text{ the set }T\setminus \h{T}
\text{ is }\h{p}\text{-small
over every }\rho\in \h{T}.
$$
Since $\t{T}\setminus T$ is $q$-small over
every $\rho\in T$,
by Lemma \ref{schlem4},
$$\text{ the set }
\t{T}\setminus \h{T}
\text{ is }
(\h{p}+q)\text{-small over every }
\rho\in \h{T}.
$$
Since  $\h{p}+q<2\h{p}$,
by Lemma \ref{schlem3}
there exists a computable subtree
$T^*$ of $\h{T}$
 such that $[T^*] = [\h{T}]$,
 $T^*$ is $p$-bushy over $\xi$,
and $\ell(T^*)$ is
$2\h{p}$-small over each $\rho $
such that $[\rho]\cap [T^*]\ne\emptyset$.
Thus $(\xi^*, T^*\cap [\xi^*]^\preceq,p, 2\h{p})$ is the desired condition
forcing $\mcal{R}_\Psi$ negatively
where $\xi^*$ is a sufficiently long
extension of $\xi$
in $ \{\rho\in T^*:
[\rho]\cap [T^*]\ne\emptyset\}$.
Thus we are done in Case 1.

\ \\

\textbf{Case 2.}
Otherwise.

By Lemma \ref{schlem2},
there exists  a computable
subtree $\h{T}$ of $\t{T}$
such that
\begin{itemize}
\item $\h{T}$
is $\h{p}$-bushy over $\eta$;

\item $\ell(\h{T})$ is
$q$-small over every $\rho$
such that $[\rho]\cap [\h{T}]\ne\emptyset$;
\item For every $Y\in [\h{T}]$,
$\Psi^Y$ is total.
\end{itemize}
Thus by Lemma \ref{schlem1},
there exists a condition
$(\eta,T^*,p^*,q)$ extending
$(\eta,\h{T},\h{p},q)$ that forces
$\mcal{R}_\Psi$. Thus we are done.

\end{proof}
\subsection{Avoid computing Schnorr random real}\label{schsec1}
\def\J{J}
Our proof employs
the framework in \cite{greenberg2011diagonally} with a
combinatorial difference. We begin with some combinatorial notions.

\begin{definition}
Given a finite set
$S$, a collection of sets
$$\mcal{B}=\{B_0,\cdots,B_{n-1}\}\subseteq \mcal{P}(S)$$
is \emph{$(k,\delta)$-hash} in $S$ if for every $\J\subseteq n$ with
$|\J| = k$, we have:
$$\big|\ \bigcap\limits_{i\in \J} B_i\ \big|/| S|
< \delta.$$
\end{definition}
Sometimes we simply say ``$(k,\delta)$-hash
collection $\mcal{B}$" when the set $S$ is clear.
\begin{lemma}\label{codinglem2}
Given any $1>\varepsilon>\delta>0$, any $k\in\omega$ with
$k>\log \delta/\log \varepsilon$,
we have: for every $n\in\omega$,
there exists an $N=N(\varepsilon,\delta,k,n)$
 such that if $| S|>N$,
then there exists a $(k,\delta)$-hash
collection of sets $\mcal{B}\subseteq \mcal{P}(S)$
with $|\mcal{B}|>n$ such that for all $B\in \mcal{B}$,
$| B|/| S|>\varepsilon$.
\end{lemma}
\begin{proof}
Let $1>\h{\varepsilon}>\varepsilon$
be such that $\h{\varepsilon}^k<\delta$.
Such $\h{\varepsilon}$ exists since
$k>\log \delta/\log \varepsilon$.
To construct  members of $\mcal{B}$,
namely $B_i,i\leq n$, let
$B_i$ include each $x\in S$ independently
of every thing else with probability
$\h{\varepsilon}$.
By law of large number,
if $|S|$ is sufficiently large, then
since $\delta >\h{\varepsilon}^k $,
we have with high probability:
\begin{align}\nonumber
\text{ for every }i\leq n, | B_i|> \varepsilon |S|
\text{ and for every }J\subseteq n
\text{ with }|J|=k,
\big|\  \bigcap\limits_{i\in J}
B_{i}\ \big|< \delta |S|.
\end{align}

In summary, if $|S|$ is sufficiently large,
then the above construction
generate a $(k,\delta)$-hash
collection
of sets $\mcal{B}$ satisfying
$$| \mcal{B}|>n\wedge
(\forall B\in\mcal{B})[ | B|/| S|>\varepsilon]$$
 with probability larger than $0$.
This means such collection of sets exists.

\end{proof}

In the following text of this subsection,
let $\t{T}$ be a computably bounded
computable tree with
$\eta$ as its stem,
let $$T = \{\rho\in \t{T}:
[\rho]\cap [\t{T}]\ne\emptyset\}.$$
Suppose
$\t{T}$ is a
%exactly
 $p$-bushy over $\eta$ and
$\ell(\t{T})$ is $q$-small over every $\rho\in T$
where $p,q$ are
computable function.
Let $\h{p}:\omega\rightarrow\mathbb{Q}$
be such a computable function that
\begin{align}\nonumber
6q<3\h{p}<p,\
\lim\limits_{n\rightarrow \infty}
q(n)/\h{p}(n)\searrow=0, \
\lim\limits_{n\rightarrow \infty}
\h{p}(n)/p(n)\searrow=0
\end{align}
 and $(\h{p},q)$ \allowsplit.
The main ingredient is the following.
\begin{lemma}\label{schlem11}
Suppose $\Psi^Y$ is total for all
$Y\in [\t{T}]$.
For every $\rho\in T$, every $\lambda>0$,
there exists a $V^*\subseteq 2^{<\omega}$ with $\m(V^*)<\lambda$
and a finite %\heighthomogeneous\
tree $\h{T}$ with
$\rho$ as its stem such that $\h{T}$ is
$\h{p}$-bushy over $\rho$ and for some $N\in\omega$,
$$(\Psi^\tau\uhr N)\downarrow\in [V^*]^\preceq
\text{ for all }
\tau\in \ell(\h{T}).$$

\end{lemma}
\begin{proof}
If such $V^*$ does not exists, then for every
$V$ with $\m(V)\geq 1-\lambda$,
we have that for most $\tau$, $\Psi^\tau\in [V]^\preceq$.
Then we select $V_0,\cdots,V_{k-1}$ for some
$k$ that is not so large (say $k=  \log \lambda/\log (1-\lambda)$)
 such that
$$\m(\cap_{j<k}V_j)<\lambda $$ while there are still many
$\tau$ such that $$\Psi^\tau\in [\cap_{j<k}V_j]^\preceq.$$
More specifically, let $S_{V_j}$ be the set of $\tau$
such that $\Psi^\tau\uhr N\in [V_j]^\preceq$.
We want to make sure that $\bigcap_j S_{V_j}$ is sufficiently big
and therefore $\bigcap_j V_j$ is the desired $V^*$.
Since for each $j$, $\overline{S}_{V_j}$ is very small,
therefore, above a sufficiently large level $l$
(depending on $k$) $\bigcap_j S_{V_j}$ must be still sufficiently
big above level $l$
(since $p(n)>k\h{p}(n)$ for all
sufficiently large $n$). But below level $l$, it may not be the case
if we don't carefully choose $V_0,\cdots,V_{k-1}$.
To this end, note that there are much more (than $k$)
sets $V$ such that $\m(V)\geq 1-\lambda$ and $S_V$
is sufficiently big. Given the level $l$, we
select $V_0,\cdots,V_{k-1}$ among a very large
(depending on $\t{T}\cap \omega^{\leq l}$)
hash collection so that the initial segments
of $S_{V_j}$ below level $l$ are identical.
%The hashness will ensure that $\m(\bigcap_j V_j)<\lambda$.

Let $k,l,M,N\in\omega$ and $\mcal{V}\subseteq \mcal{P}(2^N)$ be such that:
\begin{enumerate}
\item
$k>\frac{\log \lambda}{\log (1-\lambda)};$

\item $5k\h{p}(l)<p(l)$;

\item $M$ is sufficiently large so that
for every $M$ many trees $T_0,\cdots,T_{M-1}\subseteq \t{T}\cap \omega^{\leq l}$,
 at least $k$ of them that are identical;
%, namely
%$T_{i_0},\cdots,T_{i_{k-1}}$, such that
%$T_{i_r}$ are identical for all $r< k$;

\item $\mcal{V}\subseteq \mcal{P}(2^N)$ is
a $(k,\lambda)$-hash
collection of finite sets
 such that
$|\mcal{V}|\geq M$ and
$\m(V)>1-\lambda$ for all $V\in\mcal{V}$.

\end{enumerate}
Clearly $k$ exists. Given $k$, by the condition on
$\h{p},p$, any sufficiently large
$l$ satisfies (2). Since $\t{T}\cap \omega^{\leq l}$
is finite, any sufficiently large $M$ satisfies (3).
By Lemma \ref{codinglem2}, $N$ and the collection $\mcal{V}$ exists.
Thus the objects as above exist.

By hypothesis on $\t{T}$, there exists a sufficiently large time
$t\in\omega$ such that
for every
$\tau\in\t{T}\cap T[t]\cap \omega^{\leq t}$,
the Turing functional $\Psi^\tau\uhr N$
converges.
For every $V\subseteq 2^N$,
let $$S_V = \big\{
\tau\in [\rho]^\preceq\cap\t{T}\cap T[t]\cap \omega^t:
\Psi^\tau\uhr N\in V
\big\}.$$
Since $\t{T}/T$ is $q$-small and
$\Psi^Y$ is total for all $Y\in [\t{T}]$,
we have that $S_{2^N}$ is $(p-q)$-big over $\rho$.
If for some $V$ with $\m(V)<\lambda$,
$S_V$ is $\h{p}$-big over $\rho$, then we are done.
Suppose this is not the case.

Note that $S_{2^N} = S_{V}\cup S_{2^N\setminus V} $
for all $V\subseteq 2^N$.
Therefore,
 by Lemma \ref{schlem4},
for every $V\subseteq 2^N$ with
$\m(V)\geq 1-\lambda$,
$$S_{2^N}\setminus S_V
\text{ is }(\h{p}+q)
\text{-small over }\rho.
$$
By definition of $M$, there exist
$V_0,\cdots,V_{k-1}\in\mcal{V}$ such that the
set of predecessors of $S_{V_j} $ below level $l$, namely
$\{\sigma\in \omega^{\leq l}:[\sigma]^\preceq\cap S_{V_j}\ne\emptyset\}$,
are identical.
Let \begin{align}\nonumber
&\h{T} = \{\sigma\in \t{T}\cap  [\rho]^\preceq \cap \omega^{\leq t}:
\sigma\preceq\tau\text{ for some }\tau\in \cap_{j<k} S_{V_j}\}
\\ \nonumber
\text{ and }&V^* = \cap_{j<k} V_j.
\end{align}
 We prove that $\h{T},V^*$ are as desired.
By definition of $(k,\lambda)$-hash, $\m(V^*)<\lambda$.
Clearly
$$
\text{ for every }\tau\in \ell(\h{T}),
(\Psi^\tau\uhr N)\downarrow\in V^*.$$
%and $\h{T}$ is \heighthomogeneous.
Since below level $l$, the set of predecessors of $S_{V_j}$
are identical.
Therefore
\begin{align}\label{scheq16}
\h{T}\text{ is }(p-\h{p}-q)
\text{-bushy from level }
|\rho|\text{ to level }l-1.
\end{align}
Since $S_{2^N}\setminus S_{V_j}$ is $(q+\h{p})$-small,
by Lemma \ref{schlem4},
$$\cup_{j<k}(S_{2^N}\setminus S_{V_j})
\text{ is }k(q+\h{p})\text{-small.}$$
By Lemma \ref{schlem4} again,
$$\cap_{j<k}S_{V_j}\text{ is }
(p-q-kq-k\h{p})\text{-big above level }l.$$
Therefore the tree
\begin{align}\label{scheq15}
\h{T}\text{ is }(p-4k\h{p})\text{-bushy above level }l
\text{ since }4k\h{p}>q+kq+k\h{p}.
\end{align}
Since $p(n)-4k\h{p}(n)>\h{p}(n)$ for all $n\geq l$
and $p-\h{p}-q>\h{p}$,
combine (\ref{scheq16})(\ref{scheq15})$\h{T}$ is
$\h{p}$-bushy over $\rho$.
Thus we are done.
\end{proof}
The following is exactly the same as Lemma \ref{schlem2}.
Let $\h{p}:\omega\rightarrow\mathbb{Q}$ be a computable function
such that both $(p,\h{p})$ and $(\h{p},q)$ \allowsplit.
\begin{lemma}\label{schlem5}
Suppose for every $\rho\in T$,
every $m$,
the set $\big\{\tau\in \t{T}\cap [\rho]^\preceq:
\Psi^\tau(m)\downarrow\big\}$
is $\h{p}$-big over $\rho$.
Then there exists a computable
subtree $\h{T}$ of $\t{T}$ with
$\eta$ being its stem such that:
\begin{enumerate}
\item $\h{T}$ is $\h{p}$-bushy over $\eta$;
\item $\ell(\h{T})$ is $q$-small over each $\rho$ such that
$ [\rho]\cap [\h{T}]\ne\emptyset$;
%$\{\rho\in \h{T}: [\rho]\cap [\h{T}]\ne\emptyset\}$;
\item For every $Y\in [\h{T}]$,
$\Psi^Y$ is total.
\end{enumerate}
\end{lemma}

Now we are ready to prove our conclusion.
\begin{lemma}\label{schlem7}
Every condition
admits an extension that forces
$\mcal{R}_\Psi'$.

\end{lemma}
\begin{proof}
 Fix a condition $(\eta,\t{T},p,q)$
 and recall that $$T = \big\{
\rho\in \t{T}:[\rho]\cap [\t{T}]\ne\emptyset
\big\}.
$$
Clearly $T$  is a co-c.e. pruned
subtree of $\t{T}$
(and with $\eta$ as its stem), $T$ is $(p-q)$-bushy
over $\eta$ and $\t{T}\setminus T$ is
$q$-small over every $\rho\in T$.
 Suppose $|\eta|$ is sufficiently
large so that there exists a computable
function $\h{p}:\omega\rightarrow\mathbb{Q}$ such that
$$6q<3\h{p}<p,
\lim\limits_{n\rightarrow \infty}
q(n)/\h{p}(n)\searrow=0,
\lim\limits_{n\rightarrow \infty}
\h{p}(n)/p(n)\searrow=0$$
 and $(\h{p},q)$ \allowsplit.

\textbf{Case 1.}
There exist $\xi\in T$,
$m\in\omega$
such that
the set $\big\{\tau\in \t{T}:
\Psi^\tau(m)\downarrow\big\}$
is not $\h{p}$-big over $\xi$.

This part is the same as Case 1 of Lemma \ref{schlem6}
and is therefore omitted.

\textbf{Case 2.}
Otherwise.

For convenience, we simply assume that
$\Psi^{Y}$ is total for all $Y\in [\t{T}]$.
We inductively define a sequence of finite
trees $T_n\subseteq \t{T}$ with $\eta$ as their stem
 and a Schnorr test $V_0,V_1,\cdots$
 that will succeed on all $\Psi^Y, Y\in[\t{T}]$.
Let $T_0 = \{\eta\}$.
Suppose we have defined $T_0,\cdots, T_n$ and $V_0,\cdots,V_{n-1}$.
To define $V_n$, wait for such a time $t$
that for every $\sigma\in \ell(T_n)\cap T[t]$,
there exists a finite,
%\heighthomogeneous,
 $\h{p}$-bushy over $\sigma$ tree $T_\sigma$
with $\sigma$ as its stem and
a $V_\sigma$ with $\m(V_\sigma)<2^{-n}/|T_n|$
such that for every $\tau\in \ell(T_\sigma)$,
$$(\Psi^\tau\uhr N)\downarrow\in [V_\sigma]^\preceq \text{ for some }N.$$
By Lemma \ref{schlem11}, such time $t$ must exist.
Define
\begin{align}
&V_n = \bigcup\limits_{}\{V_\sigma:\sigma\in\ell(T_{n})\cap T[t]\}
\text{ and }\\ \nonumber
&T_{n+1} = T_n\cup (\bigcup\limits_{}
\{T_\sigma:\sigma\in \ell(T_n)\cap T[t]\}).
\end{align}
 Note that $\m(V_n)\leq 2^{-n}$, therefore $V_0,V_1,\cdots$
is a Schnorr test.
Let $\h{T} = \cup_n T_n$. It's easy to see that
the Schnorr test $V_0,V_1,\cdots$ succeed on
$\Psi^Y$  for all $Y\in [\h{T}]$.
Clearly $\h{T}$ is $\h{p}$-bushy over $\eta$ and
$$\{\rho\in \h{T}:[\rho]\cap [\h{T}]\ne\emptyset\} =
\h{T}\cap T.$$
Thus
$(\eta,\h{T},\h{p},q)$ is the desired extension
of $(\eta,\t{T},p,q)$ forcing $\mcal{R}'_\Psi$.

\end{proof}

\section{A $\DNR_h$ that is of minimal degree}

Our goal in this section is to prove the following.
Let $h\in\omega^\omega$ be an order function
such that $\lim_{n\rightarrow \infty}
h(n)/(\prod_{m<n}h(m))^k=\infty$
for all $k$.
\begin{theorem}\label{schth1}
For every $X$,
there exists a $\DNR_h^X$ that is of minimal degree.

\end{theorem}

The rest of this section will prove Theorem \ref{schth1}.
In section \ref{schsec2}, the core tree of a condition,
namely the set $\{\rho\in \t{T}:[\rho]\cap [\t{T}]\ne\emptyset\}$
is co-c.e. Since we are working
on $\DNR$ relative to an arbitrary oracle,
this is no longer true.
In \cite{{khan2017forcing}},
Khan and Miller construct a minimal degree
within a sufficiently bushy tree with an
arbitrary complex core tree.
This is done by passively wait for
$\Psi$ to split and only focus on
the nodes above which it does split.

A tree $T$ is \emph{strong c.e.} iff there exists
a computable array of finite trees $T_s,s\in\omega$ such that
$$T_s\subseteq T_{s+1}, T_{s+1}\setminus T_s\subseteq [\ell(T_s)]^\preceq
\text{ and }\cup_s T_s=T.$$
%A tree $T$ is \heighthomogeneous\ if
%$\ell(T)\subseteq \omega^n$ for some $n$.

We will, again, use the Mathias type forcing
to force all requirements.
In this section, a \emph{condition }
 is
 a tuple
 $(\eta,\t{T},T,p_T,q_T)$
 such that
 \begin{enumerate}
 \item $\t{T}$ is a strong c.e.
 tree with $\eta$ as its stem;
 \item $T\subseteq \t{T}$ is an infinite
 tree with no leaf and with $\eta$ as its stem;
 \item $p_T,q_T$ are computable functions
 from $\omega$ to $\mathbb{Q}$;

 \item $T$ is $p_T$-bushy over $\eta$
 and $\t{T}\setminus T$ is $q_T$-small over
 every $\sigma\in T$.
 \item $p_T(n)>1$ for all $n\geq |\eta|$ and
 $$\lim\limits_{n\rightarrow \infty}
 \dfrac{p_T(n)}{\max\{q_T(n),1\}\cdot|\t{T}\cap \omega^n|^k}=\infty$$
 for all $k$.
\end{enumerate}

We emphasis that there is no complexity restriction on
$T$. Again, a condition $(\eta,\t{T},T,p_T,q_T)$
is seen as a collection of candidates of the $G$ we are constructing,
namely $[T]$.
A condition $(\eta',\t{T}',T',p'_T,q'_T)$ \emph{extends}
condition $(\eta,\t{T},T,p_T,q_T)$ iff:  $\t{T}'\subseteq \t{T}$
and $T'\subseteq T$ (which automatically implies $\eta'\in T$).
The requirement is:
$$
\mcal{R}_\Psi: \Psi^G\text{ is computable or non total or }
G\leq_T \Psi^G.
$$
A condition  $(\eta,\t{T},T,p_T,q_T)$ \emph{forces}
$\mcal{R}_\Psi$ iff for every
$G\in [T]$, $\Psi^G$ satisfies
$\mcal{R}_\Psi$.

We begin with some notions.
Let $\t{T}$ be a strong c.e. tree
with $\eta$ as its stem.
For a Turing functional $\Psi$,
an $n\in\omega$,
and a computable function $q:\omega\rightarrow\mathbb{Q}$,
let
\begin{align}\label{scheq18}
&\mcal{\t{V}}_{\Psi,n,q} =
\big\{
V\subseteq 2^n:
\text{The set } \{\sigma\in \t{T}:
(\Psi^\sigma\uhr n)\downarrow\in 2^n\setminus V\}
\text{ is }q\text{-small over }\eta
\big\}. \\ \nonumber
&\mcal{V}_{\Psi,n,q} =
\big\{
V\subseteq 2^n:
\text{ there exists
a finite } q\text{-big over }\eta\text{ set }S\subseteq \t{T}
\\ \nonumber
&\hspace{3cm}\text{ such that for every }\sigma\in S,
(\Psi^\sigma\uhr n)\downarrow\in V.
\big\}.
\end{align}
Intuitively, $V\in \mcal{\t{V}}_{\Psi,n,q}$
means that $q$ is a measure of how much tree one needs to prune
in order to force $\Psi^G\in [V]$.
If $V\in \mcal{\t{V}}_{\Psi,n,q}$ for small $q$, it means
$\Psi$ easily avoid $2^n\setminus V$; we interpret this as 
the measure of $2^n\setminus V$ is smaller than $q$.
On the other hand, if $V\in \mcal{V}_{\Psi,n,q}$
for some $q$, it means we can force $\Psi^G\in [V]$
in a $\Sigma_1^0 $ way on a $q$-bushy tree;
 we interpret this as the measure of $V$ is lager than 
 $q$.
We make some simple observations. These observations
 coincides with the measure interpretation
 of $\mcal{\t{V}}_{\Psi,n,q}$.
 For example, item (2) can be interpreted as 
 if $2^n\setminus V, 2^n\setminus V'$ 
 has measure smaller than $q,q'$ respectively,
 then their union, $2^n\setminus (V\cap V')$, has measure
  smaller than $q+q'$.

\begin{lemma}\label{uemlem3}
Let $V,V'\subseteq 2^n$:
\begin{enumerate}

\item   $V\in\mcal{V}_{\Psi,n,q}$ if and only if
$2^n\setminus V\notin\mcal{\t{V}}_{\Psi,n,q}$.

\item
If $V\in \mcal{\t{V}}_{\Psi,n,q}$, $V'\in\mcal{\t{V}}_{\Psi,n,q'}$,
then $V\cap V'\in\mcal{\t{V}}_{\Psi,n,q+q'}$.

\item If $V \in\mcal{\t{V}}_{\Psi,n,q}, V'\in \mcal{V}_{\Psi,n,q'}$,
then $V\cap V'\in\mcal{V}_{\Psi,n,q'-q}$;

\item If $V \in\mcal{\t{V}}_{\Psi,n,q}$, $V'\notin \mcal{V}_{\Psi,n,q'}$,
then
 $V\setminus V'\in \mcal{\t{V}}_{\Psi,n,q+q'}$.

\item If $V\in\mcal{\t{V}}_{\Psi,n,q}$,
$V'\notin \mcal{\t{V}}_{\Psi,n,q'}$,
then $V\setminus V'\in \mcal{V}_{\Psi,n,q'-q}$.

\item If $V\notin \mcal{V}_{\Psi,n,q}$, $V'\notin \mcal{V}_{\Psi,n,q'}$,
then $V\cup V'\notin \mcal{V}_{\Psi,n,q+q'}$.

\end{enumerate}
\end{lemma}
\begin{proof}
Item (1) follows from definition.

For item (2), let
\begin{align}
&S= \{\sigma\in \t{T}:
(\Psi^\sigma\uhr n)\downarrow\in 2^n\setminus V\},
\\ \nonumber
& S'=
\{\sigma\in \t{T}:
(\Psi^\sigma\uhr n)\downarrow\in 2^n\setminus V'\}.
\end{align}
Since $V\in \mcal{\t{V}}_{\Psi,n,q}$, $V'\in\mcal{\t{V}}_{\Psi,n,q'}$,
$S,S'$ are $q$-small, $q'$-small respectively.
 By Lemma \ref{schlem4}, $S\cup S'$ is
$(q+q')$-small over $\eta$.
But $$S\cup S' = \{\sigma\in \t{T}:
(\Psi^\sigma\uhr n)\downarrow\in 2^n\setminus (V'\cap V)\}.$$
 Compare with (\ref{scheq18})
 and   we are done.

For item (3), let
$S'\subseteq \t{T}$ be a finite $q'$-big over $\eta$ set
witnessing $V'\in \mcal{V}_{\Psi,n,q'}$.
Let $$S  = \{\sigma\in S': (\Psi^\sigma\uhr n)\downarrow\notin V\}.$$
Since $V\in \t{\mcal{V}}_{\Psi,n,q}$,
$S$ is $q$-small over $\eta$.
Thus by Lemma \ref{schlem4},
$S'\setminus S$ is $(q'-q)$-big over $\eta$. And clearly
for every $\sigma\in S'\setminus S$,
$(\Psi^\sigma\uhr n)\downarrow\in V\cap V'$.
Note that for this item, it's crucial that
for some $n$, $V,V'\subseteq 2^n$.

For item (4), note that $V'\notin \mcal{V}_{\Psi,n,q'}$
implies $2^n\setminus V'\in \t{\mcal{V}}_{\Psi,n,q'}$.
Thus by item (2), $V\cap (2^n\setminus V') = V\setminus V'\in
\t{\mcal{V}}_{\Psi,n,q+q'}$.
Item (5)(6) follows from item (3)(2) respectively in a similar fashion.

\end{proof}
Suppose $T$ is an infinite subtree
of $\t{T}$ with no leaf and with $\eta$ as its stem.
Firstly, it's easy to prove that if
there are not enough successor of $\eta$ that make
$\Psi(n)$ converge, then we can force it to diverge.
\begin{lemma}\label{schlem8}
Suppose $\t{T}\setminus T$ is $q$-small
over each $\sigma\in T$
and $T$ is $p$-bushy.
\begin{enumerate}
\item
Suppose $S$ is a subset of $\t{T}$ that is
$q'$-small over $\eta$  with $p>q'+q$, then
there exists an infinite subtree
$T'$ of $T$ with no leaf and with $\eta$ as its stem
such that $\t{T}\setminus T'$ is $(q+q')$-small
over every $\sigma\in T'$
%$T'$ is $(p-q')$-bushy over $\eta$
and $[T']\cap [S] =\emptyset$.

\item
Suppose $\emptyset\in \t{\mcal{V}}_{\Psi,n,q''}$
with $p>q''$,
then there exists
an infinite subtree
$T'$ of $T$ with no leaf and with $\eta$ as its stem
such that $\t{T}\setminus T'$ is $(q+q'')$-small
over every $\sigma\in T'$
% $T'$ is $(p-q'')$-bushy over $\eta$
and for every $X\in [T]$,
$\Psi^X$ is not total.

\end{enumerate}
\end{lemma}
\begin{proof}
For item (1). Consider the following subtree of
$T$:
\begin{align}\nonumber
T' = \big\{
\sigma\in T:
[S]^\preceq\text{ is  }q'\text{-small over }\sigma
\big\}.
\end{align}
It's obvious that $[T']\cap [S] =\emptyset$.
Clearly $\eta\in T'$ and for every $\sigma\in T'$,
$\sigma$ admit at least $p(|\sigma|)-q'(|\sigma|)-q(|\sigma|)$ many immediate successor in $T'$.
Thus $T'$ is $(p-q'-q)$-bushy over $\eta$.
Since $p>q'+q$, we have that $T'$ is infinite with no leaf.
Note that for every $\sigma\in T'$, less than
$q'(|\sigma|)$ many immediate successor of $\sigma$ in $T$
is contained in $T\setminus T'$.
Therefore $T\setminus T'$ is $q'$-small over each
$\sigma\in T'$. Meanwhile, $\t{T}\setminus T$ is
$q$-small over each $\sigma\in T'$.
Thus
by Lemma \ref{schlem4},
 $\t{T}\setminus T'$ is $(q+q')$-small over each $\sigma\in T'$.

 Item (2) follows from item (1) by setting
 $S$ to be $\{\sigma\in \t{T}: (\Psi^\sigma\uhr n)\downarrow\}$.
\end{proof}

Our core argument is the following Lemma \ref{schlemmain2}
which shows that we can thin a tree by a factor of
$j$ to make $j$ many Turing functionals
split.
Let $\t{T}_i,i<j$ be strong c.e. trees with
$\eta_i,i<j$ as their stem respectively.
Fix $j$ many Turing functionals $\{\Psi_i:i<j\}$.
For any computable function $q:\omega\rightarrow\mathbb{Q}$, let
\begin{align}\nonumber%\label{scheq6}
&\mcal{\t{V}}^i_{n,q} =
\big\{
V\subseteq 2^n:
\text{The set } \{\sigma\in \t{T}_i:
(\Psi^\sigma\uhr n)\downarrow\in 2^n\setminus V\}
\text{ is }q\text{-small over }\eta_i
\big\}. \\ \nonumber
&\mcal{V}^i_{n,q} =
\big\{
V\subseteq 2^n: 2^n\setminus V\notin \mcal{\t{V}}^i_{n,q}
\big\}.
\end{align}
Suppose each $\t{T}_i$ admit an infinite subtree
$T_i$ such that $T_i$ has no leaf,
$\t{T}_i\setminus T_i$ is $q_i$-small over each $\sigma\in T_i$
and
$T_i$ is $p$-bushy over $\eta_i$
where $p,q_i$ are computable function from
$\omega$ to $\mathbb{Q}$.

\begin{lemma}\label{schlemmain2}
Let $q,q',q''$ be computable functions
from $\omega$ to $\mathbb{Q}$ such that $p>q+q'+q''$.
Suppose $V^*\in\mcal{\t{V}}_{n^*,q'}^i,\emptyset\notin \mcal{\t{V}}_{n,q'+2jq}^i$
 for all $i<j, n\in\omega$.
 Either of the following is true:
\begin{enumerate}
\item There exists
a sequence of sets $V_0,\cdots,
V_{j-1}\subseteq [V^*]^\preceq$ with $[V_i]^\preceq,i<j$
being mutually disjoint
such that for every $i<j$,
$V_i\in\mcal{V}^i_{n,q}$
for some $n$;

\item There exists a $i<j$, an
infinite tree $T\subseteq T_i$ with no leaf
with $\eta_i$ as its stem
such that
$\t{T}_i\setminus T$ is $(2jq+q'+q_i)$-small over every
$\sigma\in T$,
%$T$ is $(p-2jq-q')$-bushy
and there exists a finite $(2jq+q'+q'')$-big over $\eta_i$ set $S\subseteq \t{T}_i$
such that for
 every $X\in[T]\cap [S]$,
 $\Psi_i^X$ is not total;

\item There exists a $i<j$, an
infinite tree $T\subseteq T_i$ with no leaf
with $\eta_i$ as its stem
such that
$\t{T}_i\setminus T$ is $(2jq+q'+q''+q_i)$-small over every
$\sigma\in T$
%$T$ is $(p-2jq-q'-q'')$-bushy
and for
every $X\in [T]$,
if $\Psi_i^X$ is total, then
$\Psi_i^X$ is computable.

\end{enumerate}

\end{lemma}
\begin{remark}
The intuition of this Lemma is the following.
For $V\in \mcal{V}^i_{n,q}$, we can see $q$
as a measure of $V$ assigned by $\Psi_i$.
Meanwhile, it's trivial to verify that
for $j$ many continuous positive measure $\mu_0,\cdots,\mu_{j-1}$
on a space $S$, there exists a partition of $S$
into $k$ disjoint pieces, namely $S_0,\cdots,S_{j-1}$
such that $\mu_i(S_i)\geq 1/k$.
Here a measure is  continuous means it does not
assign non zero measure on singleton.
In our application, $V^*=\bot$, $q'\equiv 0$,
 $q,q''$ are mush smaller compared
to $p$ and much larger than $q_i$.
The hypothesis $\emptyset\notin \mcal{\t{V}}_{n,q'+2jq}^i$
simply means there are  many strings $\sigma$
(a measure of more than $q'+2jq$)
such that $(\Psi^\sigma_i\uhr n)\downarrow$.

\end{remark}
\begin{proof}
We prove by induction on $j$.
For $j=1$, $V^*\in \mcal{\t{V}}_{n^*,q'}^0,
\emptyset\notin \mcal{\t{V}}^0_{n^*,q+q'}$ implies
by Lemma \ref{uemlem3} item (5) that
$V^*\in \mcal{V}^0_{n^*,q}$. Thus the conclusion follows when
$j=1$.
 Now assume that
conclusion holds for $j-1$
by fulfilling item (1).
For every $n>n^*,i<j$,
let $V^*_n = [V^*]^\preceq\cap 2^n$.
As said in the remark, the major trouble maker is
singletons, therefore
let
 $$W^i_{n,q} =\{\rho\in V^*_n: \{\rho\}\in\mcal{V}^i_{n,q}\}.$$
 The key part is Case 1 where not so much measure
 is put on singletons. i.e., for some sufficiently
 big set $S$, some $\t{i}$, $\Psi_{\t{i}}^\sigma\uhr n\notin \cup_{i<j}W^i_{n,q}$
 for all $\sigma\in S$.

\ \\

\textbf{Case 1}.
For some $\t{i}<j$, $n>n^*$, $\cup_{i<j}W^i_{n,q}\notin \mcal{\t{V}}^{\t{i}}_{n,q+q'}$.

Intuitively, the hypothesis of this case says
that the measure  of $2^n\setminus (\cup_{i<j}W^i_{n,q}) $
is larger than $q+q'$.
In this case we
locate a $\h{V}$   so that every proper
subset of $\h{V}$
can be very easily avoided by all $\Psi_i$ and for some $\h i$,
$\h{V}$ can not be as easily avoided by $\Psi_{\h{i}}$
as its proper subset.
We argue that  by Lemma \ref{uemlem3},
$\h{V}$ can  be avoided by all $\Psi_i$. Thus we force $\Psi_{\h{i}}^G\in [\h{V}]$
while we force $\Psi_i^G\notin [\h{V}]$ for all $i\ne \h i$. And the conclusion follows
 by induction.

Since $V^*\in \mcal{\t{V}}^{\t{i}}_{n^*,q'}$
(the measure of $2^{n^*}\setminus V^*$ is smaller than $q'$),
we have $V^*_n\in \mcal{\t{V}}^{\t{i}}_{n,q'}$ for all $n>n^*$.
Therefore, the hypothesis of Case 1 implies,
the measure of $(2^n\setminus \cup_{i<j}W^i_{n,q})
\setminus (2^n\setminus V^*_n)$
is larger than  $(q+q')-q' = q$.
More precisely,
by Lemma \ref{uemlem3} item (5), $$V^*_n\setminus(\cup_{i<j}W^{i}_{n,q})\in \mcal{V}^{\t{i}}_{n,q}.$$
Consider the collection of such sets
$$V\subseteq  V^*_n\setminus(\cup_{i<j}W^{i}_{n,q})
\text{ that for some }\h{i}<j, V\in \mcal{V}^{\h{i}}_{n,q}.$$
Let
 $\h{V}\subseteq V^*_n\setminus(\cup_{i<j}W^{i}_{n,q})$ be the minimal
 (in the sense of subset)
 among this collection.
i.e.,
$$\text{ there is no }V'\subsetneq \h{V}\text{ such that }
\exists i'<j [V'\in \mcal{V}^{i'}_{n,q}].$$
Suppose $V\in \mcal{V}^{\h{i}}_{n,q}$. It is not necessary that $\h{i}=\t{i}$.

We  show that $V^*_n\setminus \h{V}\in\mcal{\t{V}}^{i}_{n,q'+2q}$ for all $i<j$.
This is because every proper subset of $\h{V}$
can be easily avoided by minimality of $\h{V}$.
So, simply choose two subsets of $\h V$ and apply Lemma
\ref{uemlem3} item (6) to avoid both of them.
Let $\rho\in \h{V}$ be arbitrary (which clearly exists since
$\h{V}\in \mcal{V}^{\h{i}}_{n,q}$) and let $i<j$.
Since $\rho\notin \cup_{i<j}W^{i}_{n,q}$, so
$\{\rho\}\notin \mcal{V}^i_{n,q}$.
Since $\h{V}$ is minimal, so $\h{V}\setminus\{\rho\}\notin
\mcal{V}^i_{n,q}$. Thus by Lemma \ref{uemlem3} item (6),
$\h{V}\notin \mcal{V}^{i}_{n,2q}$. This implies, by Lemma
\ref{uemlem3} item (4), $V^*_n\setminus \h{V}\in \mcal{\t{V}}^{i}_{n,q'+2q}$.

Now the conclusion follows by induction where $V^*,q'$ are
reset to be $V^*_n\setminus \h{V}$ and $q'+2q$ respectively,
 and $j$ is reduced to $j-1$. In the end, if item (1) holds,
 then $\h{V}$ will be the set corresponding to $\h{i}$.

 \ \\

\textbf{Case 2}. There exists a $i<j$ such that
$|W^i_{n,q}|$ is not bounded with respect to  $n$.

Let $B = \{i<j:|W^i_{n,q}|\text{ is not bounded with respect to }n \}$.
It is clear that for  some sufficiently large $n$,
there exists a $\rho_i\in W^i_{n,q}$ for each $i\in B$
such that $\rho_i,i\in B$ are mutually different and
$\rho_i\notin W^{\h{i}}_{n,q}$ for all $\h{i}\notin B$.
Let $W =\{\rho_i: i\in B\}$.
Since
each $\rho_i$ has measure smaller than $q$ with respect to $\Psi_{\h i}$)
therefore $(2^n\setminus V^*_n)\cup W$
has measure smaller than $q$ with respect to $\Psi_{\h i}$ when $\h i\notin B$.

More precisely, since $\{\rho_i\}\notin \mcal{V}^{\h{i}}_{n,q}$
for all $i\in B,\h{i} \in j\setminus B$,
by Lemma \ref{uemlem3}
item (4),
$V^*_n\setminus W\in \mcal{\t{V}}^{\h{i}}_{n,q'+|B|q}$ for all
$\h{i}\notin B$. Thus the conclusion follows by induction
where $j$ is reduced to $j-|B|$ and
$V^*,q'$ are reset to be $V^*_n\setminus W$, $q'+|B|q$ respectively.
In case item (1) holds,
the set corresponding to $\Psi_i$ with $i\in B$ will be
$\{\rho_i\}$.

\ \\

\textbf{Case 3}. Otherwise.

Let $W_{n,q} = \cup_{i<j}W^i_{n,q}$
and $W_q = \cup_n W_{n,q}$.
It is clear that $W_{n,q}\subseteq [W_{n-1,q}]^\preceq$.
Therefore $W_q$ is a tree (in $2^{\geq n^*}$).
Moreover, since it is not Case 2,
for some $u$ and $\overline{n}>n^*$, we have $|W_{n,q}|
=u$ for all $n>\overline{n}$.
By compactness,
$W_{n,q}$ is c.e., therefore
$W_q$ is a c.e. tree. Since $|W_{n,q}|= u$ for all $n>\overline{n}$,
$W_q$ is a computable tree and every element in $[W_q]$ is computable.
Fix an $i$, to force  $\Psi^G$ to be a member of $[W_q]$ (fulfilling item (3)),
consider the set on which $\Psi^G\notin [W_q]$:
 $$\h{S} = \{\sigma\in \t{T}_i:
(\Psi_i^\sigma\uhr n)\downarrow \notin W_{n,q}
\text{ for some }n>\overline{n}\}.$$
Note that for every $n$, since it is not Case 1,
$\Psi_i$ can easily avoid $2^n\setminus W_{n,q}$.
But it doesn't mean $\Psi_i$ can easily avoid $\overline{W}_q$.
i.e., $\h{S}$ might not be very small.
We show that if we can not easily avoid $\overline{W}_q$,
then we can force $\Psi_i$ to be non total.

If $\h{S}$ is $(q+q'+q'')$-small over $\eta$,
then by Lemma \ref{schlem8}, there exists
an
infinite tree $T\subseteq T_i$ with no leaf
with $\eta_i$ as its stem
such that
$$\t{T}_i\setminus T\text{ is }(q+q'+q''+q_i)\text{-small over every }
\sigma\in T$$
% $T$ is $(p-q-q'-q'')$-bushy
and $[T]\cap [\h{S}]=\emptyset$,
which means for
every $X\in [T]$, every $n>\overline{n}$,
 $$(\Psi_i^X\uhr n)\downarrow
 \rightarrow \Psi_i^X\uhr n \in W_{n,q}.$$
Then we are done since for
  every $X\in [T]$,
$\Psi_i^X$ is total implies $\Psi_i^X\in [W_q]$ and is therefore
computable, which fulfills
 item (3) of this lemma.

If $\h{S}$ is $(q+q'+q'')$-big over $\eta$,
by definition of ``bigness", there exists a finite
 subset
$S$ of $\h{S}$ that is $(q+q'+q'')$-big
over $\eta$.
Let $\h{n}$ be sufficiently large, say,
$$\h{n} >  \max\{n: \text{ for some }\sigma\in S,
(\Psi_i^\sigma\uhr n)\downarrow\notin W_{n,q}\}.$$
Since it is not Case 1,
$W_{\h{n},q}\in \t{\mcal{V}}^i_{\h{n},q+q'}$ for all $i<j$.
By Lemma \ref{schlem8},
there exists  an infinite
subtree $T$ of $T_i$ with no leaf and
with $\eta$ as its stem
such that
\begin{align}\nonumber
&\t{T}_i\setminus T\text{ is
}
(q+q'+q_i)\text{-small over every }\sigma\in T
%$T$ is $(p-q-q')$-bushy
\text{ and }\\ \nonumber
&(\Psi_i^\sigma\uhr \h{n})\downarrow\rightarrow
\Psi_i^\sigma\uhr \h{n}\in W_{\h{n},q}\text{ for all }\sigma\in T.
\end{align}
We show that $S,T$ is the desired pair fulfilling
conclusion (2).
Let $X\in [T]\cap [S]$.
By definition of $S$,
$(\Psi^X\uhr n)\notin W_{n,q}$ for some $n<\h{n}$,
which means,
$$(\Psi^X\uhr \h{n})\downarrow\rightarrow
(\Psi^X\uhr \h{n})\notin W_{\h{n},q}$$
since $W_{n,q}\subseteq [W_{n-1,q}]^\preceq$.
But by definition of $T$,
$$(\Psi^X\uhr \h{n})\downarrow\rightarrow
(\Psi^X\uhr \h{n})\in W_{\h{n},q}.$$
Thus $\Psi^X$ can not be total.

\end{proof}

Fix a Turing functional $\Psi$
and a condition $(\eta,\t{T},T,p_T,q_T)$.
Suppose  $\h{p}:\omega\rightarrow \mathbb{Q}$ is
such a computable function that
 $$\lim\limits_{n\rightarrow \infty}
 \dfrac{p_T(n)}{\max\{\h{p}(n),1\}\cdot|\t{T}\cap \omega^n|^k}=\infty,
  \lim\limits_{n\rightarrow \infty}
 \dfrac{\h{p}(n)}{\max\{q_T(n),1\}\cdot|\t{T}\cap \omega^n|^k}=\infty$$
 for all $k$ (which means both
 $(\h{p},q_T)$ and $(p_T,\h{p})$ satisfy item (5) of
 the definition of condition)
 and for every $n\geq |\eta|$,
 $$p_T(n)>4\h{p}(n)\cdot |\t{T}\cap \omega^n|>16q_T(n)\cdot|\t{T}\cap \omega^n|.$$
 For convenience,
we transform Lemma \ref{schlemmain2}
into the following:

\begin{lemma}\label{schlem9}
Either there exists an extension of
$(\eta,\t{T},T,p_T,q_T)$ forcing $\Psi^G$
to be computable or non total,
or for any $m$, any
 set $B=\{\eta_i\}_{i<j}\subseteq T\cap \omega^m$,
 there exists   for each $i\in B$
 a finite %\heighthomogeneous\
 tree $T'_i\subseteq \t{T}$
 with $\eta_i$ being its stem
  such that $T'_i$ is $\h{p}/2$-bushy
 over $\eta_i$ and for every $\sigma\in \ell(T'_i)$,
  $(\Psi^\sigma\uhr n)\downarrow\in V_i$ for some $n$
 where $[V_i]^\preceq,i\in B$ are mutually disjoint.

\end{lemma}
\begin{proof}

Fix a $B = \{\eta_i\}_{i<j}\subseteq T\cap \omega^m$.
Let $\t{T}_i=\t{T}\cap [\eta_i]^\preceq$.
Note that $$p_T(n)-2j\h{p}(n)-q_T(n)>1$$ for all $n\geq |\eta|$ since
$p_T(n)>4\h{p}(n)\cdot |\t{T}\cap \omega^n|$.
Therefore, if for some $i<j$, $n\in\omega$,
$\emptyset\in \t{\mcal{V}}^i_{n,2j\h{p}}$,
then by Lemma \ref{schlem8},
 there exists an infinite subtree
$T'$ of $T_i$ with no leaf and with $\eta_i$ as its stem
such that $\t{T}_i\setminus T'$ is $(q_T+2j\h{p})$-small
over every $\sigma\in T'$
%$T'$ is $(p_T-2j\h{p}-q_T)$-bushy over $\eta_i$
and for every $X\in [T']$,
$\Psi^X$ is not total.
Thus $(\eta_i,\t{T}_i,T',p_T, q_T+2j\h{p})$
is the desired extension.

Suppose this is not the case.
  Apply Lemma \ref{schlemmain2}
with
\begin{align}
&\Psi_i = \Psi^{\eta_i},
\t{T}_i = \t{T}\cap [\eta_i]^\preceq,
T_i = T\cap [\eta_i]^\preceq,
\\ \nonumber
&n^*=0, V^*=\{\bot\}, q'\equiv 0,
q=\h{p}, q_i= q_T, q'' = q_T,
\end{align} we have that one of the three items
 holds.

 If item (1) of Lemma \ref{schlemmain2} holds,
 it means that there exists finite set
 $S_i\subseteq \t{T}_i$
 for each $i<j$,
  with $S_i$ being $\h{p}$-big over
 $\eta_i$ respectively, such that
  for every $\sigma\in S_i$,
  $(\Psi^{\sigma}\uhr n)\downarrow\in V_i$
  for some $n$ where $[V_i]^\preceq,i<j$
  are mutually disjoint.
  Since $\t{T}_i\setminus T_i$ is $q_T$-small over
  every $\sigma\in T_i$ with $q_T<\h{p}/2$,
  there exists  for each $i<j$ a finite %\heighthomogeneous\
  tree $T'_i\subseteq \t{T}_i$ with $\eta_i$ as its stem
  such that $T'_i$ is $\h{p}/2$-bushy
and for every $\sigma\in \ell(T'_i)$,
there exists $n$ such that
 $(\Psi^\sigma\uhr n)\downarrow\in V_i$.
 Clearly such $T'_i,i<j$ are computable (if exists)
 uniformly in $B,\t{T},p_T,q_T,\h{p}$.
 Thus we are done for this case.

 If item (2) of Lemma \ref{schlemmain2} holds, it means
 that for some $i<j$,  there exists an
infinite tree $T'\subseteq T_i$ with no leaf
with $\eta_i$ as its stem
such that
$\t{T}_i\setminus T'$ is $(2j\h{p}+q_T)$-small over every
$\sigma\in T'$,
%$T'$ is $(p_T-2j\h{p})$-bushy over $\eta_i$
and there exists a finite $(2j\h{p}+q_T)$-big over $\eta_i$ set $S\subseteq \t{T}_i$
such that for
 every $X\in[T']\cap [S]$,
 $\Psi_i^X$ is not total.
 But $\t{T}_i\setminus T'$ is $(2j\h{p}+q_T)$-small over
 $\eta_i$ and
 $S$ is
 $(2j\h{p}+q_T)$-big over
 $\eta_i$. This implies that $S\cap T'\ne\emptyset$.
 Suppose $\eta'\in S\cap T'$.
 Clearly the condition
 $$(\eta',\t{T}_i\cap [\eta']^\preceq, T'\cap [\eta']^\preceq,
 p_T,2j\h{p}+q_T)$$ is a desired
 extension of $(\eta,\t{T},T,p_T,q_T)$ that forces
 $\mcal{R}_\Psi$.

 If item (3) of Lemma \ref{schlemmain2} holds, the conclusion follows similarly. Thus we are done.

\end{proof}

Now it's easy to prove the following:
\begin{lemma}\label{schlem10}
Every condition admit an extension forcing
$\mcal{R}_\Psi$.

\end{lemma}
\begin{proof}
Fix a condition
$(\eta,\t{T},T,p_T,q_T)$.
Without loss of generality suppose
$|\eta|$ is sufficiently large
so that the function $\h{p}$ defined before
Lemma \ref{schlem9} exists
(otherwise extend $\eta$ to be so).
Suppose   there is  no extension
of $(\eta,\t{T},T,p_T,q_T)$
forcing $\Psi^G$ to be computable or non total.
We construct a strong c.e. tree
by  inductively define a sequence of finite trees $T_n,n\in\omega$
so that $\Psi$ split on the tree. Which ensure that
 for every infinite path $G$ of the tree, $G\leq_T \Psi^G$.
Let $T_0=\{\eta\}$,
$T_1 = \t{T}\cap \omega^{|\eta|+1}$.
Suppose we have defined $T_0,\cdots,T_n$.
Let $A_n $ denote the set of string $\sigma$
such that at some step before $n$, some nodes have grown above
 $\sigma$ and after that point, the tree above $\sigma$ has remained
 unchanged. i.e.,
\begin{align}\nonumber
A_n = \big\{\sigma\in T_n: & \text{ for some }m\leq n,
\sigma\in \ell(T_{m-1})\text{ and}\\ \nonumber
&(T_m\setminus T_{m-1})\cap [\sigma]^\preceq= (T_n\setminus T_{m-1})\cap [\sigma]^\preceq
\supsetneq \{\sigma\}\big\}.
\end{align}

Moreover, we assume that for every $\sigma\in A_n$,
$\ell(T_n\cap[\sigma]^\preceq)\subseteq \omega^l$ for some
$l\in\omega$. This is because we can require in Lemma \ref{schlem9}
that each $T'_i$ being \heighthomogeneous\, i.e., its leaves
lies on the same level.

At step $n$, we wait for a time that
 above some $\sigma\in A_n$,
$\Psi$ split by that time.
For every $\sigma\in A_n$,
let $S_\sigma = \ell(T_n)\cap [\sigma]^\preceq$.
Wait for such a time $t$ that there exists
a $\sigma\in A_n$
a subset $B_\sigma=\{\eta_i\}_{i<j}$ of $S_\sigma$
with $S_\sigma\setminus B_\sigma$ being
$q_T$-small over every $\tau\in \{\tau\in [\sigma]^\preceq:
\text{ for some }i<j, \tau\preceq \eta_i \}$
such that for every $i<j$,
there exists a finite \heighthomogeneous\ tree
$T_i'\subseteq \t{T}\cap [\eta_i]^\preceq$ with
$\eta_i$ as its stem such that
\begin{itemize}
\item $T_i'\supsetneq \{\eta_i\}$ is $\h{p}/2$-bushy over $\eta_i$;
\item For every $\tau\in \ell(T'_i)$,
there exists an $n$ such that
 $(\Psi^\tau\uhr n)\downarrow\in V_i$.
 \end{itemize}
 Where $[V_i]^\preceq,i<j$ are mutually disjoint.
 By Lemma \ref{schlem9} such time $t$ exists.
 In this case we say the tree grows due to $\sigma$
 at step $n$.

 Let $T_{n+1} = T_n\cup(\cup_{i<j}T'_i)$
 and declare $S_\sigma\setminus B_\sigma$ leaves of
 $\t{T}'$.
 Let $$\t{T}' = \cup_n T_n\text{ and }
 T' = T\cap \{\rho\in \t{T}':[\rho]\cap [\t{T}']\ne\emptyset\}.$$
 We show that $(\eta,\t{T}',T',\h{p}/2,2q_T)$
 is the desired extension.
 It's easy to see that $\Psi^Y$ computes $Y$
 if $Y\in [T']$ since we have made sure that
 $\Psi$ split on $T'$.
 It's also trivial to verify item (1)(2)(3)(5) of
 definition of condition.
 By the construction of $\t{T}'$,
 it's direct that $\t{T}'$ is $\h{p}/2$-bushy
 over $\eta$.

Now we verify item (4) of the definition of condition,
which is simple but tedious.
Let $D$ denote the set of nodes that is declared to be leaves
of $\t{T}'$.
It's easy to see that for every $n$, every $\sigma\in A_n$,
$D$ is $q_T$-small over $\sigma$.
We prove that this is true for all $\tau\in T'$.
\begin{claim}\label{schclaim2}
For every $\tau\in T'$, $D$ is $q_T$-small over $\tau$.
\end{claim}
\begin{proof}
Fix a  $\tau\in T'$. Note that there must exist
a time point at which $\tau$ is put into $\t{T}'$.
i.e., there must exist
some $m$,
$\sigma\in \ell(T_{m})$, such that
$\tau\in (T_{m+1}\setminus T_{m})\cap [\sigma]^\preceq$.
Since $[\tau]\cap [\t{T}']\ne\emptyset$,
the tree must grow again due to $\sigma$,
i.e., there exists an $n$ with $\sigma\in A_n$
(which means the tree above $\sigma$ remains unchanged before
step $n$) such that
$$(T_{n+1}\setminus T_n)\cap [\sigma]^\preceq\ne
(T_{m+1}\setminus T_{m})\cap [\sigma]^\preceq.$$
Note that if the tree above $\tau$ does not grow at
step $n$,
%$$(T_{n+1}\setminus T_n)\cap [\tau]^\preceq=
%(T_{m+1}\setminus T_{m})\cap [\tau]^\preceq,$$
then
all elements in  $[\tau]^\preceq \cap \ell(T_n)$ will be declared
leaf of $\t{T}'$ and the tree above $\tau$ will no longer grows,
which makes it impossible that $[\tau]\cap [\t{T}']\ne\emptyset$.
Therefore we have: $$[\tau]^\preceq\cap B_\sigma\ne\emptyset.$$
This implies,
by our construction, $S_\sigma\setminus B_\sigma$
is $q_T$-small over $\tau$. Meanwhile,
it's obvious that $B_\sigma\subseteq A_{n+1}$. Therefore
$D$ is $q_T$-small over every $\eta'\in B_\sigma$.
Thus $D$ is $q_T$-small over $\tau$.

\end{proof}

Next we argue that:
\begin{claim}\label{schclaim1}
 For every $\rho\in T$,
if $\rho\notin T'$, then it must be the case that
at some point $ n$, all elements in $\ell(T_n)\cap [\rho]^\preceq$
have been declared leaves. i.e.,
\begin{align}\label{scheq14}
\t{T}'\setminus T'\subseteq
\{\rho\in \t{T}:
\t{T}'\cap [\rho]^\preceq\text{ is finite and }
\ell(
\t{T}'\cap [\rho]^\preceq)\subseteq D\}\cup  (\t{T}\setminus T).
\end{align}
\end{claim}
\begin{proof}
The key point is:
\begin{align}\label{scheq11}
&\text{ for every }\zeta\in T,
\text{ if } \zeta
\text{ ever enters some }A_n,
\text{ then }[\zeta]\cap [\t{T}']\ne\emptyset.
\end{align}
Simply because each $B_\sigma$ still contains many elements in
$T$ if $\sigma\in T$.
Fix a $\rho\in \t{T}'\setminus T'$
such that $\rho\in T$.
Suppose otherwise that (\ref{scheq14}) is not true. By definition of $T'$,
$[\rho]\cap [\t{T}']=\emptyset$.
Note that there must exist
an $n$ and a $\sigma'\in A_n$ such that
$\rho\in (T_n\setminus T_{n-1})\cap [\sigma']^\preceq$.
Clearly $\sigma'\in T$ since $\sigma'\preceq\rho$.
Note that for every $\sigma\in T$, if $\sigma$ ever
enters $A_n$ at some step $n$, then there must exists some point
that the tree grows due to $\sigma$.
Which means at some step $n$, $B_{\sigma'}$ is defined.
Since $[\rho]^\preceq\cap B_{\sigma'}\ne\emptyset$
(otherwise all $\ell([\rho]^\preceq\cap T_n)$ will be declared leaves,
a contradiction to the otherwise assumption),
therefore
 $S_{\sigma'}\setminus B_{\sigma'}$
is $q_T$-small over $\rho$.
Combine with $\t{T}\setminus T$ being $q_T$ small over
$\rho$ and $S_{\sigma'}$ being $\h{p}/2$-big
over $\rho$, we have $$[\rho]^\preceq \cap B_{\sigma'}\cap T\text{ is }
(\h{p}/2-2q_T)\text{-big over }\rho.$$
Therefore  there exists a $\tau\in [\rho]^\preceq\cap B_{\sigma'}\cap T\ne\emptyset$.
But $B_{\sigma'}\subseteq A_n$,
so $\tau\in A_n$.
As we observed in (\ref{scheq11}), $[\tau]\cap [\t{T}']\ne\emptyset$,
which implies $[\rho]\cap [\t{T}']\ne\emptyset$,
a contradiction.
\end{proof}
Now we are ready to check item (4).
For every $\tau\in T'$,  $$\t{T}'\setminus T'\subseteq
(\t{T}\setminus T)\cup (T\setminus T')$$
and $(\t{T}\setminus T)$ is $q_T$-small over $\tau$
(since $\tau\in T'\subseteq T$).
On the other hand, if $T\setminus T'$ is not $q_T$-small over
$\tau$, then by Claim \ref{schclaim1},
$D$ is not $q_T$-small over $\tau$,
a contradiction to Claim \ref{schclaim2}.
Thus $\t{T}'\setminus T'$ is $2q_T$-small over $\tau$
and we are done.

\end{proof}

\begin{proof}[Proof of Theorem \ref{schth1}]

The initial condition is
$(\bot, h^{<\omega},T,h_1,2)$
where
$$T = \{\sigma\in h^{<\omega}: \text{ for every }n\leq |\sigma|,
\sigma(n)\ne \Psi^X_n(n)\text{ if }\Psi^X_n(n)\downarrow\}$$
 and $h_1(n) = h(n+1)$.
It's easy to verify that $(\bot, h^{<\omega},T,h_1,2)$
is indeed a condition.
By Lemma \ref{schlem10}, there exists
a sequence of conditions
$d_0\supseteq d_1\supseteq\cdots$ where
$d_s = (\eta_s,\t{T}_s,T_s,p_s,q_s)$
such that every requirement is forced at some point.
Thus $G=\cup_s \eta_s$ is a member of all
conditions $d_s$, therefore $G$
satisfy all requirements.
Clearly $G\in \DNR^X_h$ since $G\in d_0$. Thus we are done.

\end{proof}

\bibliographystyle{amsplain}
\bibliography{F:/6+1/Draft/bibliographylogic}
%\bibliography{bibliographylogic}

\end{document}